\documentclass{amsart}
\usepackage{graphicx}
\newcommand{\mathsmaller}[1]{#1}
\usepackage{amssymb}
\usepackage{bold-extra}
\usepackage{float}
\newcommand{\C}{\mathbb C}

\newcommand{\F}{\mathbb F}
\newcommand{\R}{\mathbb R}
\newcommand{\N}{\mathbb N}
\newcommand{\Z}{\mathbb Z}

\newcommand{\Span}{\mathrm{span }}
\newcommand{\supp}{\mathrm{supp }}

\newcommand{\Supp}{\mathrm{supp}}

\newtheorem{thm}{Theorem}[section]
\newtheorem{pro}[thm]{Proposition}
\newtheorem{cor}[thm]{Corollary}
\newtheorem{lem}[thm]{Lemma}
\theoremstyle{definition}
\newtheorem{rem}[thm]{Remark}
\newtheorem{defn}[thm]{Definition}

\usepackage[hidelinks]{hyperref}
\begin{document}


\title[\tiny A computational approach to the Thompson group $F$]{\large  A \MakeLowercase{computational approach to the} T\MakeLowercase{hompson group} $F$}

\author[\tiny S. Haagerup\qquad U. Haagerup\qquad M. Ramirez-Solano]{ S. Haagerup\,\,\,\,\,\,\, U. Haagerup$^*$\,\,\,\,\,\,\, M. Ramirez-Solano$^\dag$}
\date{February 5, 2015}

\thanks{$^*$ Supported by the ERC Advanced Grant no. OAFPG 247321, and partially supported by the Danish National Research Foundation (DNRF) through the Centre for Symmetry and Deformation at the University of Copenhagen, and the Danish Council for Independent Research, Natural Sciences.}
\thanks{$^\dag$ Supported by the ERC Advanced Grant no. OAFPG 247321, and by the Center for Experimental Mathematics at the University of Copenhagen.}

\newcommand{\Addresses}{{
  \bigskip
  \footnotesize
  S\o ren Haagerup, \textsc{Elmebakken 15, 5260 Odense S, Denmark.}\par\nopagebreak
  \textit{E-mail address}: \texttt{soren@artofweb.dk}

  \medskip
  \medskip

  Uffe Haagerup, \textsc{Department of Mathematical Sciences, University of Copenhagen, Universitetsparken 5, 2100 K\o benhavn \O, Denmark.}\par\nopagebreak
  \textit{E-mail address}: \texttt{haagerup@math.ku.dk}

  \medskip
  \medskip

  Maria Ramirez-Solano, \textsc{Department of Mathematical Sciences, University of Copenhagen, Universitetsparken 5, 2100 K\o benhavn \O, Denmark.}\par\nopagebreak
  \textit{E-mail address}: \texttt{mrs@math.ku.dk}
}}

\begin{abstract}
Let $F$ denote the Thompson group with standard generators $A=x_0$, $B=x_1$. It is a long standing open  problem whether $F$ is an amenable group.
By a result of Kesten from 1959, amenability of $F$ is equivalent to
$$(i)\qquad ||I+A+B||=3$$ and to
$$(ii)\qquad ||A+A^{-1}+B+B^{-1}||=4,$$
where in both cases the norm of an element in the group ring $\C F$ is computed in $B(\ell^2(F))$ via the regular representation of $F$.
By extensive numerical computations, we obtain precise lower bounds for the norms in $(i)$ and $(ii)$, as well as good estimates of the spectral distributions of $(I+A+B)^*(I+A+B)$ and of $A+A^{-1}+B+B^{-1}$ with respect to the tracial state $\tau$ on the group von Neumann Algebra $L(F)$.
Our computational results suggest, that
$$||I+A+B||\approx 2.95 \qquad ||A+A^{-1}+B+B^{-1}||\approx 3.87.$$
It is however hard to obtain precise upper bounds for the norms, and our methods cannot be used to prove non-amenability of $F$.
\end{abstract}


\maketitle

\section{\textbf{Introduction}}\label{s:one}

\begin{defn}

  The Thompson group $F$ is the group of homeomorphisms $g:[0,1]\to[0,1]$ for which:
  \begin{itemize}
    \item the endpoints satisfy $g(0)=0$, $g(1)=1$,
    \item it is piecewise linear with finitely many break points on the diadic numbers $\Z[\tfrac 12]\cap (0,1)$,
    \item all slopes of $g$ are in the set $2^{\Z}:=\{2^n\mid n\in\Z\}$.
  \end{itemize}
\end{defn}
$F$ is a countable group, and it is generated by the elements $A,B$ whose graphs are shown in Fig.~\ref{f:generatorsABofF}.
\begin{figure}[h]
\centerline{
\includegraphics[scale=.50]{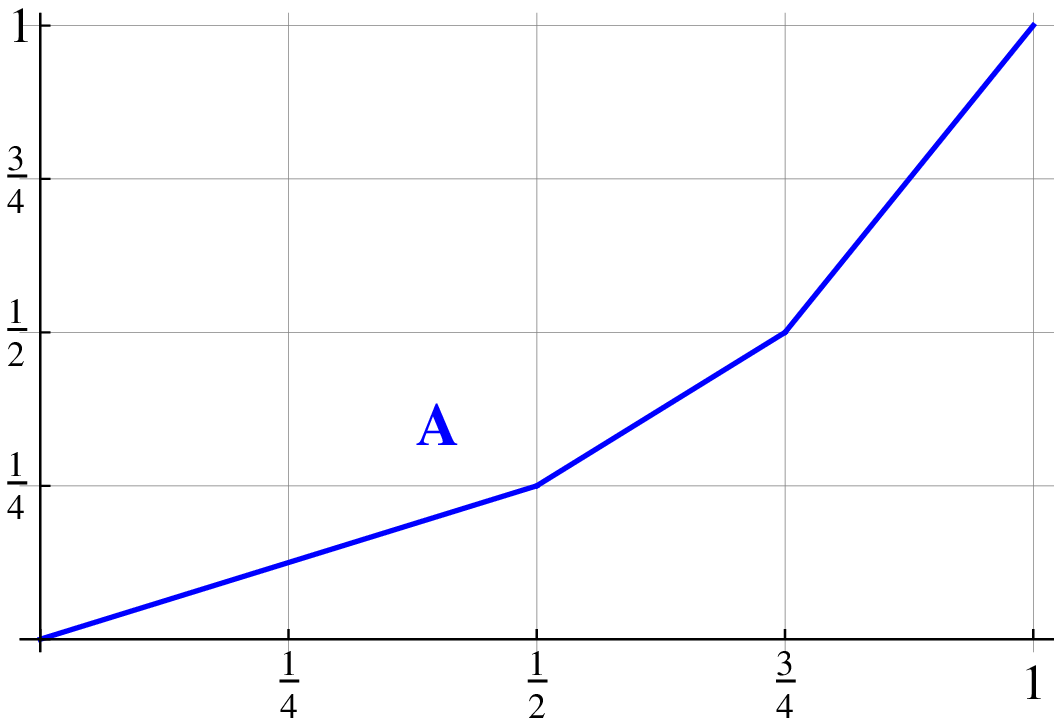}\qquad
\includegraphics[scale=.50]{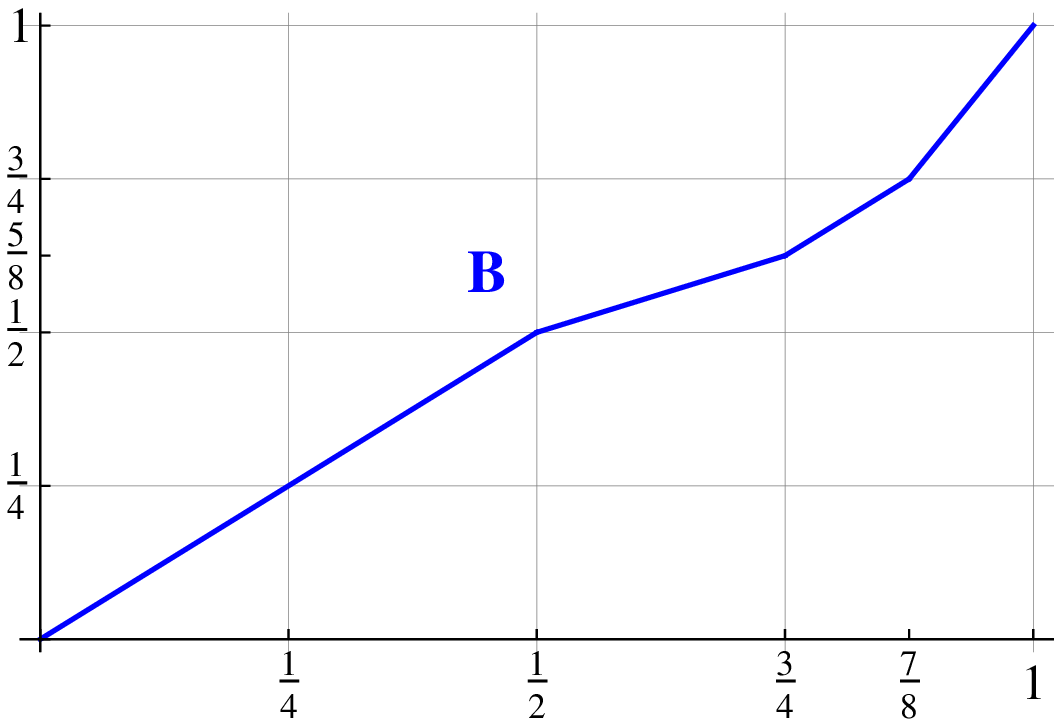}
}
\vspace*{8pt}
  \caption{The generators $A,B$ of $F$.\label{f:generatorsABofF}}
\end{figure}

Moreover, it has a finite presentation in terms of $A,B$, namely
$$F=\langle A,B\mid [AB^{-1}, A^{-1}BA]=[AB^{-1},A^{-2}BA^2]=e\rangle,$$
where the group commutator $[g,h]=ghg^{-1}h^{-1}$ is defined as usual.

Recall that elementary amenability implies amenability, and a copy of the free group $\F_2$ (on two generators) inside a group implies non-amenability of the group.
It is known that $F$ is not elementary amenable, i.e.  $F$ cannot be obtained from finite or Abelian groups by taking subgroups, quotients, extensions, and direct limits.  On the other hand, by a result of Brin and Squier \cite{BrinS}, $F$ does not contain a copy of $\F_2$.
For more information on the Thompson group $F$, see the survey paper by Cannon, Floyd and Perry \cite{CFP}.

It is a main open problem to decide whether the Thompson group $F$ is amenable.  Recently, Monod \cite{Monod} has constructed examples of groups of homeomorphisms of $[0,1]$ which are non-amenable, but which also do not contain a copy of $\F_2$. These groups resemble $F$.  Moreover, Olesen and the second named author has shown in \cite{HO} that if the reduced $C^*$-algebra $C_r^*(T)$ of the (non-amenable) Thompson group $T$  is simple, then $F$ is non-amenable.
Both of the above mentioned results suggest that $F$ might not be amenable, and extrapolations of our computational results point in the same (non-amenability) direction.

The present paper grew out of an attempt to test the amenability problem for $F$ by using computers to estimate norms of certain elements in the group ring $\C F$ of $F$. By the norm $||a||$ (see also Section~\ref{s:two}) of an element $a$ in the group ring of a discrete group $\Gamma$ we mean
\begin{equation}
  ||a||=||\lambda(a)||_{B(\ell^2(\Gamma))},
\end{equation}
where $\lambda$ is the left regular representation of $\Gamma$.
As explained in Section~\ref{s:two}, it is standard to write  $a\in B(\ell^2(\Gamma))$ instead of $\lambda(a)\in B(\ell^2(\Gamma))$, for any $a\in \C F$, and we will continue with this tradition.
Our starting point is the following two theorems due to Kesten and Lehner: (See Section~\ref{s:six} for a more detailed discussion).

\begin{thm}[\cite{Ke2},\cite{Leh}]\label{t:1.2}

  Let $\Gamma$ be a discrete group with a generating set $X=\{s_1,\ldots,s_k\}$ such that $k\ge 2$ and $e\not\in X$. Then,
  $$2\sqrt k \le ||e+s_1+\dots s_k||\le k+1.$$
  Moreover, the upper bound is attained if and only if $\Gamma$ is amenable, and the lower bound is attained if and only if $X$ generates $\Gamma$ freely.
\end{thm}

\begin{thm}[\cite{Ke2},\cite{Kesten}]\label{t:1.3}

  Let $\Gamma$ be a discrete group with a generating set $X=\{s_1,\ldots,s_k\}$ such that $k\ge 2$ and $X\cap X^{-1}=\emptyset$. Then,
  $$2\sqrt {2k-1 }\le ||s_1+\dots s_k+s_1^{-1}+\cdots+s_k^{-1}||\le 2k.$$
  Moreover, the upper bound is attained if and only if $\Gamma$ is amenable, and the lower bound is attained if and only if $X$ generates $\Gamma$ freely.
\end{thm}
Hence for the Thompson group $F$ we get

\begin{cor}\label{c:1.4}

Let $A$ and $B$ be the standard generators of $F$, and let $I$ denote the unit element of $F$. Then
    $$2\sqrt2<||I+A+B||\le 3,$$
    $$2\sqrt 3<||A+A^{-1}+B+B^{-1}||\le 4.$$
      Moreover, in both cases the upper bound is attained if and only if $F$ is amenable.
\end{cor}

Let $L(\Gamma)$ denote the von Neumann algebra of a discrete group $\Gamma$, i.e. $L(\Gamma)$ is the von Neumann algebra in $B(\ell^2(\Gamma))$ generated by $\lambda(\Gamma)$. Then
\begin{equation}
  \tau(T)=\langle T\delta_e,\delta_e\rangle, \quad T\in L(\Gamma)
\end{equation}
defines a normal faithful tracial state on $L(\Gamma)$ (See e.g. Section 6.7 in \cite{KadRin}). Hence
\begin{equation}
  ||T||=||T^*T||^{1/2}=\lim_{n\to\infty} \tau((T^*T)^n)^{\frac 1{2n}}
\end{equation}
(cf.~Section~\ref{s:four}).
Hence if we knew all the numbers
$$m_n(T^*T):=\tau((T^*T)^n),\quad n\in\N_0$$
we could also compute the norm $||T||$. In practice we can only compute a finite number of the moments $m_n(T^*T)$.

In this paper, we develop efficient methods, both mathematically and computationally to compute the numbers $m_n(T^*T)$  in the case
$$T=\sum_{x\in Y} x \quad \in \C \Gamma $$
for any finite set $Y$ in a discrete group $\Gamma$.
We then apply the methods to the elements $T_1$ and $T_2$ in the group ring $\C F$ of the Thompson group $F$ given by
\begin{equation}\label{e:T1andT2}
  T_1:=I+A+B, \qquad T_2:=A+A^{-1}+B+B^{-1}.
\end{equation}
As a result, we have been able to compute the moments $m_n(T_1^*T_1)$ for $0\le n\le 37$ and the moments $m_n(T_2^*T_2)=m_{2n}(T_2)$ for $0\le n\le 24$.

Using the spectral theorem to the self-adjoint operators
$$\tilde T_i:=\left(
                                                                                 \begin{array}{cc}
                                                                                   0 & T_i^* \\
                                                                                   T_i & 0 \\
                                                                                 \end{array}
                                                                               \right)\in M_2(B(\ell^2(F))),$$
(cf.~Section~\ref{s:two}), one gets that there are unique probability measures $\mu_1,\mu_2$ to $\R$ with $\supp(\mu_i)\subset [-||T_i||,||T_i||]$ such that $\mu_i$ is invariant under the reflection $t\mapsto -t$, and such that
\begin{equation}\label{e:4plus}
  \int_{-||T_i||}^{||T_i||} t^{2n}\, d\mu_i(t)=m_n(T_i^*T_i), \qquad n\in \N_0,\,\, i=1,2.
\end{equation}

Moreover, $\pm||T_i||\in\Supp(\mu_i)$.

Using methods from the theory of orthonormal polynomials applied to these two measures (cf.~Section~\ref{s:four}) we obtain from our moment calculations good lower bounds for $||T_i||$, $i=1,2$, namely
$$||I+A+B||\ge 2.86759$$
and
$$||A+A^{-1}+B+B^{-1}||\ge 3.60613.$$
In fact for each $n\le 37$ (resp. $n\le 24$), we find a lower estimate of $||T_1||$ (resp. $||T_2||$), and a suitable extrapolation of those two finite series of numbers suggests that the actual norms are much closer to 3 (resp 4), namely
$$||I+A+B||\approx2.95 \qquad \text{and} \qquad ||A+A^{-1}+B+B^{-1}||\approx3.87. $$
Furthermore, based on our moment calculations, we have also been able to estimate the Lebesgue densities of the measures $\mu_1$ and $\mu_2$ with fairly high precision. This shows that the measures $\mu_1$ and $\mu_2$ are very close to zero on the interval $[2.9,3]$  and $[3.7,4]$ respectively, but we cannot rule out, that the measures have very ``thin tails" stretching all the way up to $3$ and $4$, respectively, which would imply that $F$ is amenable
(cf.~Corollary~\ref{c:1.4}).
In comparison, one gets for the free group $\F_2$ on two generators $a,b$ that
$$||e+a+b||=2\sqrt 2\approx 2.82824$$ and
$$||a+a^{-1}+b+b^{-1}||=2\sqrt 3\approx 3.46410.$$

The measures $\mu_1$, $\mu_2$ based on $a,b$ instead of $A,B$,  will be denoted by $\mu_i^{\mathrm{free}}$ ($i=1,2$), and they can be computed explicitly (See Section~\ref{s:six}):
\begin{eqnarray*}
&&  \mu_1^{\mathrm{free}}= \frac 3{2\pi} \frac{\sqrt{8-x^2}}{9-x^2} 1_{[-2\sqrt 2,2\sqrt2]}(x) \, dx\\
&&    \mu_2^{\mathrm{free}}= \frac 2{\pi} \frac{\sqrt{12-x^2}}{16-x^2} 1_{[-2\sqrt 3,2\sqrt3]}(x) \, dx.
\end{eqnarray*}

Since 2007 a number of papers has been published about computational approaches to problems related to the Thompson group $F$, including the amenability problem (cf.~\cite{BurilloClearyWiest},\cite{GubaArzhantsevaLustingPreaux},\cite{ElderFusyRechnitzerCountingGeodesics},\cite{ElderRechnitzerWongCogrowth},\cite{ElderRechnitzerJanse}).
Our paper is the first that considers the moments of $T_1^*T_1$ for $T_1:=I+A+B$. In \cite{BurilloClearyWiest} Burillo, Cleary and Wiest use probabilistic methods to estimate the moments of $m_n(T_2^*T_2)=m_{2n}(T_2)$ for $T_2:=A+A^{-1}+B+B^{-1}$ for $n=10,20,30,\ldots,160$.
Moreover, Elder, Rechnitzer and Wong compute in \cite{ElderRechnitzerWongCogrowth} the first 22 cogrowth coefficients of $F$ with respect to the symmetric set of generators $\{A,A^{-1},B,B^{-1}\}$, from which the exact values of $m_{2n}(T_2)$ $n=1,\ldots, 11$ can easily be computed.
We will comment in more detail on the results of \cite{BurilloClearyWiest} and \cite{ElderRechnitzerWongCogrowth} at the end of Section~\ref{s:four}.

The Thompson group $F$ has been the key motivation for the research presented in this paper. However, we should stress, that the methods developed in this paper can be applied to do similar computations in any other finitely generated group, provided there is a reasonable fast algorithm to decide whether two words in the generators correspond to the same element of the group.

The rest of the paper is organized as follows:
Section~\ref{s:two} contains background material on the reduced $C^*$-algebra $C_r^*(\Gamma)$ and the group von Neumann algebra $L(\Gamma)$ associated to a discrete group $\Gamma$. In Section~\ref{s:three}, we consider for any given finite set $Y$ in a discrete group $\Gamma$, the element
$$h=\sum_{s\in Y} s$$
and explain how one can compute the moments
$$m_n=\tau((h^*h)^n),\quad n\ge 1$$
from computing much smaller number $||h_n||_2^2$, $\xi_n$, $\eta_n$, $\zeta_n$ ($n\ge 1$) associated to the pair $(\Gamma, Y)$. We next apply this to the special case of $\Gamma=F$, the Thompson group $F$, and $Y=\{I,A,B\}$ or $Y=\{A,A^{-1},B,B^{-1}\}$ in order to compute
numerically the moments $m_n$ for $1\le n\le 37$ in the first case and for $n\le n\le 24$ for the second case. In Sections~\ref{s:four} and \ref{s:fourB} we apply methods from the theory of orthogonal polynomials to estimate the norms $||I+A+B||$ and $||A+A^{-1}+B+B^{-1}||$, and to estimate the probability measures $\mu_1$, $\mu_2$, associated to $T_1=I+A+B$ and $T_2=A+A^{-1}+B+B^{-1}$ via formula~(\ref{e:4plus}) above.  The sections~\ref{s:five} and \ref{s:six} are the main theoretical sections of the paper.  In the general setting of a finite subset $Y$ of a discrete group $\Gamma$, we derive in Section~\ref{s:five} the formulas (used in Section~\ref{s:three}), that allow us to pass back and forth between the 5 sequences of numbers $||h_n||_2^2$, $\xi_n$, $\eta_n$, $\zeta_n$  and $m_n$ ($n\ge 1$).
In Section~\ref{s:six} we formulate and explain Kesten's and Lehner's results (Theorem~\ref{t:1.2} and Theorem~\ref{t:1.3}) in the setting of Leinert sets (cf.~Theorem~\ref{t:LowerboundLeinerSetsUpperBoundAmenability}).
Moreover, we make comparison between our numbers $||h_n||_2^2$, $\xi_n$, $\eta_n$, $\zeta_n$ ($n\ge 1$) and the cogrowth coefficients due to Cohen \cite{C} and Grigorchuk \cite{Gr}, and show that amenability of the group $\Gamma_0\subset \Gamma$ generated by $Y^{-1}Y$ can be decided from the asymptotics of each of the 4 sequences mentioned above (cf.~Corollary~\ref{c:7.9}).

\section{\textbf{Preliminaries}}\label{s:two}

Let $\Gamma$ be a countable discrete group. We consider the left regular representation $\lambda$ of $\Gamma$ on $\ell^2(\Gamma)$ i.e.
$$(\lambda(x)f)(y):=f(x^{-1}y),\qquad f\in \ell^2(\Gamma),\,\, x,y\in\Gamma.$$
We will also use the letter $\lambda$ for the extension of the regular representation to the group ring $\C\Gamma$ of $\Gamma$ i.e.
$$\lambda(\sum_{x\in Y} c_x x):= \sum_{x\in Y} c_x \lambda(x),$$
for any finite subset $Y\subset \Gamma$ and any set of complex numbers $(c_x)_{x\in \Gamma}$ indexed by $\Gamma$.
The reduced $C^*$-algebra of $\Gamma$ is the $C^*$-algebra generated by $\lambda(\Gamma)$, i.e.
$$C_r^*(\Gamma)=\overline{\lambda(\C\Gamma)}^{\mathrm{\, norm}},$$
and the group von Neumann algebra $L(\Gamma)$ is the von Neumann algebra generated by $\lambda(\Gamma)$, i.e.
$$L(\Gamma)=\overline{\lambda(\C\Gamma)}^{\mathrm{\, so}},$$
where the closure is taken in the strong operator topology on $B(H)$.
Note that $L(\Gamma)$  can also be expressed as $\lambda(\Gamma)'' $ (double commutant).
Moreover,
$$C_r^*(\Gamma)\subset L(\Gamma)\subset B(\ell^2(\Gamma))$$
with isometric inclusions.
We define the norm of an element $a\in\C \Gamma$ by
$$||a||=||\lambda(a)||_{B(\ell^2(\Gamma))}=||\lambda(a)||_{L(\Gamma)}=||\lambda(a)||_{C_r^*(\Gamma)}.$$
Let $\delta_x(y):=\delta_{x,y}$ for $x,y\in\Gamma$. Then $\{\delta_x \mid x\in \Gamma\}$ form an orthonormal basis for $\ell^2(\Gamma)$.
The functional
\begin{equation}\label{e:taudefinition}
  \tau(T)=\langle T\delta_e,\delta_e\rangle,\quad T\in L(\Gamma)
\end{equation}
is a normal faithful tracial state on $L(\Gamma)$. In particular,
$$\tau(ST)=\tau(TS),\quad S,T\in L(\Gamma)$$
$$\tau(T^*T)\ge 0 ,\quad T\in L(\Gamma)$$
$$\tau(T^*T)=0\Rightarrow T=0, \quad T\in L(\Gamma)$$
$$\tau(I)=1.$$
Since $\lambda:\C\Gamma\to C_r^*(\Gamma)$ is one-to-one, we may consider $\C \Gamma$ as a subalgebra of $C^*_r(\Gamma)\subset L(\Gamma)$.
Then the trace $\tau$ restricted to $\C\Gamma$ satisfies
$$\tau( \sum_{x\in Y} c_x x)=\left\{
                               \begin{array}{ll}
                                 c_e, & \text{ if } e\in Y \\
                                 0, & \text{otherwise}.
                               \end{array}
                             \right.
$$
In particular, $\tau(e)=1$ and $\tau(x)=0$ for $x\in\Gamma\backslash\{e\}$.
Moreover, for $a=\sum_{x\in \Gamma } c_x x \in \C\Gamma$,
\begin{equation}\label{e:tauastara}
  \tau(a^*a)=||a\delta_e||^2=||\sum_{x\in \Gamma} c_x\delta_x||^2=\sum_{x\in\Gamma} |c_x|^2.
\end{equation}
For all the above see e.g. Section 6.7 of  \cite{KadRin} or Section 2.5 of \cite{BrownOzawa}.
Let $S\in L(\Gamma)_{\mathrm{sa}}$ be a self-adjoint operator in $L(\Gamma)$.
Then $$\Lambda:C(\sigma(S))\to\C$$ given by
$\Lambda(f)=\tau(f(S))$ is a positive functional on $C(\sigma(S))$ such that $\Lambda(1)=1$.
Hence, there is a unique probability measure $\mu_S$ on $\sigma(S)$ for which
$$ \tau(f(S))=\int_{\sigma(S)} f d\mu_S,\qquad \forall f\in C(\sigma(S)).$$
Since $\tau$ is a faithful trace,
 $\Supp(\mu_S)=\sigma(S).$
We will also consider $\mu_S$ as a probability measure on all of $\R$ concentrated on $\sigma(S)$.
Since $||S||=r(S)$, the spectral radius of $S$, we have
\begin{equation}\label{e:normSABSalphaABSBeta}
  ||S||=\max\{|t|\mid t\in\supp(\mu_S)\}=\max\{|\alpha|,|\beta|\}
\end{equation}
 where $\alpha=\min(\supp(\mu_S))$ and $\beta=\max(\supp(\mu_S))$.
By the moments of $S$ we mean the numbers
$$m_n(S):=\tau(S^n),\qquad n\in\N_0=\N\cup\{0\}.$$
Note that $(m_n(S))_{n=0}^\infty$ are also the moments of the measure $\mu_S$, i.e.
\begin{equation}
  m_n(S)=\int_{-\infty}^\infty t^n d\mu_S(t), \qquad n\in\N_0,
\end{equation}
and since $\mu_S$ has a compact support, $\mu_S$ is uniquely determined by its moments.
For a not necessarily self-adjoint element $T\in L(\Gamma)$, it is convenient to consider the self-adjoint operator
$$\tilde T=\left(
             \begin{array}{cc}
               0 & T^* \\
               T & 0 \\
             \end{array}
           \right)\in M_2(L(\Gamma))=L(\Gamma)\otimes M_2(\C),
$$
and its moments
$$m_n(\tilde T):=\tilde \tau (\tilde T^n), \quad n\in \N_0,$$
and spectral distribution $\mu_{\tilde T}$ with respect to the normal faithful trace $\tilde \tau$  on $M_2(L(\Gamma))$ given by
$\tilde \tau=\tau\otimes \tau_2$, where $\tau_2=\frac 12 Tr$ on $M_2(\C)$, i.e.
\begin{equation}\label{e:tautilde}
 \tilde \tau \left(
                \begin{array}{cc}
                  T_{11} & T_{12} \\
                  T_{21} & T_{22} \\
                \end{array}
              \right)=\frac {\tau(T_{11})+\tau(T_{22})}2, \qquad [T_{ij}]\in M_2(L(\Gamma)).
\end{equation}

Using the trace properties of $\tau$, one has
\begin{equation}\label{e:mntautilde}
  m_n(\tilde T)=\left\{
                  \begin{array}{ll}
                    \tau((T^*T)^{n/2}), & n \text{ even} \\
                    0, & n\text{ odd.}
                  \end{array}
                \right.
\end{equation}

Since all the odd moments vanish, $\mu_{\tau_{\tilde T}}$ is symmetric. i.e. $\check{\mu}_{\tilde T}=\mu_{\tilde T}$, where
$$\check{\mu}_{\tilde T}(B)=\mu_{\tilde T}(-B), \quad B \text{ Borel set in } \R.$$
In particular,  $\min(\Supp(\mu_{\tilde T}))= -\max(\Supp(\mu_{\tilde T}))$ and hence by (\ref{e:normSABSalphaABSBeta})
\begin{equation}\label{e:normeqmaxofsupp}
  ||T||=||\tilde T||=\max (\Supp(\mu_{\tilde T})),
\end{equation}
and moreover
\begin{equation}\label{e:pmnormTinSupport}
  \pm ||T||\in\Supp(\mu_{\tilde T}).
\end{equation}
Note also that $\mu_{T^*T}$ is equal to the image measure $\phi(\mu_{\tilde T})$  of $\mu_{\tilde T}$ by the map $\phi:t\to t^2$
because $\mu_{T^*T}$ and $\phi(\mu_{\tilde T})$ are both compactly supported and by (\ref{e:mntautilde}) they have the same moments.

\section{\textbf{Computational methods and Results}}\label{s:three}

Let $\Gamma$ be a discrete group and let $Y\subset \Gamma$ be a finite set with $|Y|=q+1$ elements ($q\ge 2$).
Our main example will be $\Gamma=F$, the Thompson group, and $Y$ either $\{I,A,B\}$ or $\{A,A^{-1},B,B^{-1}\}$, but we will for some time stick to the general case in order to treat the two particular cases in a similar way. Our goal is to compute as many as possible of the moments
$$m_n:=\tau((h^*h)^n), \qquad n\in\N_0,$$
where $h=\sum_{s\in Y} s\in \C\Gamma$ and $\tau$ is the trace on $\C\Gamma$ coming from the group von Neumann algebra $L(\Gamma)$ as in Section~\ref{s:two}.
Recall that
$$\tau(x)=\left\{
            \begin{array}{ll}
              1, & x=e \\
              0, & x\ne e.
            \end{array}
          \right.
$$
Hence $m_0=1$, and for $n\ge 1$, the numbers $m_n\in\N$ are
\begin{equation}\label{e:mncomp}
  m_n=|\{(s_1,\ldots,s_{2n})\in Y^{2n}\mid s_1^{-1}s_2\cdot\cdots\cdot s_{2n-1}^{-1}s_n=e \}|,
\end{equation}
where $|X|$ denotes the number of elements in a set $X$.
For composing elements of the Thompson group $F$, we used the Belk and Brown forest algorithm from \cite{BelkBrownForestDiagrams}.
Consider the subsets $\tilde E_k\subset E_k\subset  Y^{k}$
\begin{equation}\label{e:Ekcomp}
  E_k:=\{(s_1,\ldots,s_{k})\in Y^{k}\mid s_1\ne s_2\ne \cdots \ne s_k \}
\end{equation}
\begin{equation}\label{e:tildeEkcomp}
  \tilde E_k:=\{(s_1,\ldots,s_{k})\in Y^{k}\mid s_1\ne s_2\ne \cdots \ne s_k\ne s_1 \}.
\end{equation}
We say that $\tilde E_k$ is ``cyclic" as all the cyclic rotations of each element in $\tilde E_k$ are contained in the set itself.
Since these subsets are smaller, it takes less time to compute the ``reduced" numbers
\begin{equation}\label{e:rhoncomp}
  \eta_n=|\{(s_1,\ldots,s_{2n})\in E_{2n}\mid s_1^{-1}s_2\cdot\cdots\cdot s_{2n-1}^{-1} s_{2n} =e\}|
\end{equation}
and even less time to compute the ``cyclic" numbers
\begin{equation}\label{e:zetancomp}
  \zeta_n=|\{(s_1,\ldots,s_{2n})\in \tilde E_{2n}\mid s_1^{-1}s_2\cdot\cdots\cdot s_{2n-1}^{-1} s_{2n} =e\}|.
\end{equation}
The relationship between these numbers and the moments are derived in Section~\ref{s:five}. In particular, we have
\begin{eqnarray*}\label{e:3.4}
  \zeta_n&=&\eta_n-(q-1)(\eta_{n-1}+\eta_{n-2}+\cdots+\eta_1),\qquad n\in\N\\
   m_n&=&\binom{2n}{n} q^n+\sum_{k=1}^{n}\binom{2n}{n-k} (\zeta_k+1-q)q^{n-k},\qquad n\in\N
\end{eqnarray*}
which shows that the first $n$ moments ($m_1$,\ldots, $m_n$) can be computed from either $(\eta_1,\ldots,\eta_n)$ or $(\zeta_1,\ldots,\zeta_n)$.
In the case of a symmetric set $Y$, i.e. $Y=Y^{-1}$, the reduced numbers $\eta_n$  are the even co-growth coefficients $\eta_n=\gamma_{2n}$ introduced by Cohen \cite{C} and Grigorchuk \cite{Gr}. In the notation of Elder-Rechnitzer-Wong \cite{ElderRechnitzerWongCogrowth}, $m_n=r_{2n}$ and $\eta_n=\alpha_{2n}$. (Still for a symmetric set $Y\subset \Gamma).$
In Cohen words, $\eta_n$ is the number of ``reduced" words in $Y$ of length $2n$ which represent the unit element in $\Gamma$.

It is much faster to compute the $\eta_n$'s or $\zeta_n$'s, than to compute the moments $m_n$ directly from (\ref{e:mncomp}).
However, we found a different approach to compute the moments, which speeded up the computations much further:
Let
\begin{equation}\label{e:3.5}
h_n=\left\{
        \begin{array}{ll}
          \sum\limits_{(s_1,\ldots,s_n)\in E_n} s_1^{-1}s_2\ldots s_{n-1}^{-1}s_n, &  (n \text{ even})\\
            \sum\limits_{(s_1,\ldots,s_n)\in E_n} s_1s_2^{-1}\ldots s_{n-1}^{-1}s_n, &(n \text{ odd})
        \end{array}
      \right.
\end{equation}
and recall that $||a||_2=\tau(a^*a)^{1/2}$ is the 2-norm on $L(\Gamma)$ associated with the trace $\tau$. Then
\begin{equation}\label{e:halfnumberscomp}
  \tau(h_n^*h_n)=||h_n||^2_2,
\end{equation}
and the reduced number $\eta_n$ can be computed from the square of the 2-norm of  $h_n$ by the following two equations
\begin{eqnarray}
  \eta_n&=&\xi_n-(q-1)(\xi_{n-1}+\xi_{n-2}+\cdots+\xi_1),\\
\nonumber \xi_n&=&||h_n||_2^2-(q+1)q^{n-1}
\end{eqnarray}
(cf.~Section~\ref{s:five}), and hence the moment series $(m_n)_{n=1}^\infty$ can also be computed from the numbers $(||h_n||_2^2)_{n=1}^\infty$.
In practice we computed $||h_n||_2^2$ as follows. Note first that
\begin{equation}
h_n=\sum_{x\in Y_n} c_x^{(n)}x
\end{equation}
where $Y_n\subset \Gamma$ is the set of all distinct terms in the sum~(\ref{e:3.5}), and $c_x^{(n)}\in\N$ is the multiplicity of the occurrence of $x\in Y_n$ in the sum~(\ref{e:3.5}). Note that $Y_1=Y$, $Y_2\subset Y^{-1}Y$, $Y_3\subset YY^{-1}Y$, $Y_4\subset Y^{-1}YY^{-1}Y$, etc.
Since
$$\tau(x^{-1}y)=\left\{
                  \begin{array}{ll}
                    1, & x=y \\
                    0, & x\ne y
                  \end{array}
                \right.
$$
we have
\begin{equation}
  ||h_n||_2^2=\sum_{x\in Y_n} (c_x^{(n)})^2.
\end{equation}

The advantage of computing the squared 2-norms  $||h_n||_2^2$  instead of the reduced numbers $\eta_n$ or the cyclic numbers $\zeta_n$ is that we only have to consider $(1+q)q^{n-1}$ words of length $n$ instead of $(1+q)q^{2n-1}$ words of length $2n$.
This made it possible for us to almost double the number of moments we could compute in the two cases $Y=\{I,A,B\}$ and $Y=\{A,B,A^{-1},B^{-1}\}$ for the Thompson group $F$.
The only drawback was that we need first to store all the terms of the sum~(\ref{e:3.5}), and next to sort the list in order to compute the multiplicities $c_x^{(n)}\in \N$, $x\in Y_n$.


We wrote two programs, both using parallel computing, to calculate the squared 2-norms $||h_n||_2^2$, which can be downloaded at:
\\\\
\url{https://github.com/shaagerup/ThompsonGroupF/}\\
\url{https://github.com/mariars/ThompsonGroupF/}\\
\url{http://www.math.ku.dk/~haagerup/ThompsonGroupF/}\\
\url{http://www.math.ku.dk/~mrs/ThompsonGroupF/}\\

\begin{table}[b]
\caption{The series of numbers for $h=I+A+B$ (Case 1).\label{t:Case1SeriesOfNumbers}}
{\footnotesize
\begin{tabular}{lllll}
\hline
$n$ & $||h_n||_2^2$ & $\eta_n$ & $\zeta_n$ & $m_n(h^*h)$\\
\hline
1	 & 	3	  & 	0	 & 	0	 & 	3\\
2	 & 	6	  & 	0	 & 	0	 & 	15\\
3	 & 	12	  & 	0	 & 	0	 & 	87\\
4	 & 	24	  & 	0	 & 	0	 & 	543\\
5	 & 	48	  & 	0	 & 	0	 & 	3543\\
6	 & 	96	  & 	0	 & 	0	 & 	23823\\
7	 & 	192	  & 	0	 & 	0	 & 	163719\\
8	 & 	400	  & 	16	 & 	16	 & 	1144015\\
9	 & 	800	  & 	16	 & 	0	 & 	8100087\\
10	 & 	1656	  & 	72	 & 	40	 & 	57971735\\
11	 & 	3344	  & 	104	 & 	0	 & 	418640071\\
12	 & 	7032	  & 	448	 & 	240	 & 	3046373007\\
13	 & 	14272	  & 	656	 & 	0	 & 	22314896087\\
14	 & 	30544	  & 	2656	 & 	1344	 & 	164407579407\\
15	 & 	63120	  & 	4688	 & 	720	 & 	1217526417687\\
16	 & 	137264	  & 	15712	 & 	7056	 & 	9057960864015\\
17	 & 	292160	  & 	33344	 & 	8976	 & 	67667981453831\\
18	 & 	651960	  & 	100984	 & 	43272	 & 	507425879338551\\
19	 & 	1435808	  & 	232872	 & 	74176	 & 	3818200408513415\\
20	 & 	3310592	  & 	671848	 & 	280280	 & 	28821799875573303\\
21	 & 	7593024	  & 	1643688	 & 	580272	 & 	218200189786794855\\
22	 & 	18161528	 & 	4619168	 & 	1912064	 & 	1656415132760705871\\
23	 & 	43488112	 & 	11784224	 & 	4457952	 & 	12606151256856370471\\
24	 & 	107764880	 & 	32572880	 & 	13462384	 & 	96166410605134544815\\
25	 & 	268721056	 & 	85764176	 & 	34080800	 & 	735237884585469467543\\
26	 & 	686850128	 & 	235172192	 & 	97724640	 & 	5632983879577272289359\\
27	 & 	1769246208	 & 	630718144	 & 	258098400	 & 	43241777428163458121799\\
28	 & 	4640551024	 & 	1732776752	 & 	729438864	 & 	332564656181337723832623\\
29	 & 	12254456800	 & 	4706131504	 & 	1970016864	 & 	2562203165206920141303479\\
30	 & 	32773003720	 & 	12970221624	 & 	5527975480	 & 	19773205160010752987777543\\
31	 & 	88160278544	 & 	35584492728	 & 	15172024960	 & 	152837887007013006956440295\\
32	 & 	239251904104	 & 	98515839744	 & 	42518879248	 & 	1183157961642417140248556303\\
33	 & 	652453973392	 & 	272466004928	 & 	117953204688	 & 	9172380845538923831902240519\\
34	 & 	1790526123576	 & 	758084181720	 & 	331105376552	 & 	71206765648586031626111809367\\
35	 & 	4933923852880	 & 	2110955787448	 & 	925892800560	 & 	553521536480845628126004101879\\
36	 & 	13660080583776	 & 	5903188665464	 & 	2607169891128	 & 	4308220957036953495382444267287\\
37	 & 	37952694315360	 & 	16535721813272	 & 	7336514373472	 & 	33572939291063083015187615095255\\
\hline
\end{tabular}}
\end{table}

\begin{table}[h]
\caption{ The series of numbers for $h=A+B+A^{-1}+B^{-1}$  (Case 2).\label{t:Case2SeriesOfNumbers}}
{\footnotesize
\begin{tabular}{lllll}
\hline
$n$ & $||h_n||^2_2$ & $\eta_n$ & $\zeta_n$ & $m_n(h^*h)$\\
\hline
1	 & 	4	 & 	0	 & 	0	 & 	4\\
2	 & 	12	 & 	0	 & 	0	 & 	28\\
3	 & 	36	 & 	0	 & 	0	 & 	232\\
4	 & 	108	 & 	0	 & 	0	 & 	2092\\
5	 & 	344	 & 	20	 & 	20	 & 	19884\\
6	 & 	1076	 & 	64	 & 	24	 & 	196096\\
7	 & 	3500	 & 	336	 & 	168	 & 	1988452\\
8	 & 	11324	 & 	1160	 & 	320	 & 	20612364\\
9	 & 	38708	 & 	5896	 & 	2736	 & 	217561120\\
10	 & 	134880	 & 	24652	 & 	9700	 & 	2331456068\\
11	 & 	497616	 & 	117628	 & 	53372	 & 	25311956784\\
12	 & 	1906356	 & 	531136	 & 	231624	 & 	277937245744\\
13	 & 	7747484	 & 	2559552	 & 	1197768	 & 	3082543843552\\
14	 & 	32825220	 & 	12142320	 & 	5661432	 & 	34493827011868\\
15	 & 	145750148	 & 	59416808	 & 	28651280	 & 	389093033592912\\
16	 & 	668749196	 & 	290915560	 & 	141316416	 & 	4420986174041164\\
17	 & 	3164933480	 & 	1449601452	 & 	718171188	 & 	50566377945667804\\
18	 & 	15314270964	 & 	7269071976	 & 	3638438808	 & 	581894842848487960\\
19	 & 	75551504916	 & 	36877764000	 & 	18708986880	 & 	6733830314028209908\\
20	 & 	378261586048	 & 	188484835300	 & 	96560530180	 & 	78331435477025276852\\
21	 & 	1918303887820	 & 	972003964976	 & 	503109989256	 & 	915607264080561034564\\
22	 & 	9831967554120	 & 	5049059855636	 & 	2636157949964	 & 	10750847942401254987096\\
23	 & 	50870130025336	 & 	26423287218612	 & 	13912265601668	 & 	126768974481834814357308\\
24	 &      265393048436340	 &	139205945578944  & 	73848349524776	 &  1500753741925909645997904\\
\hline
\end{tabular}}
\end{table}

The first program was written in Haskell, where the code loops through all the terms in the sum~(\ref{e:3.5}) and saves them to the hard disk.  We then use the GNU-sort program to find the multiplicities. This program was run in a supercomputer with 32 cores and 128 GB memory at the University of Copenhagen. The second program was written in $C\#$. For the case $Y=\{A,B,A^{-1},B^{-1}\}$,  we made one further step to reduce both the computation time and the size of the storing data.
Since $|E_n|=4\cdot 3^{n-1}$ in this case, the computation time will increase at least by a factor of $3$ when going from $n$ to $n+1$. In practice the computing time increased by a factor $\approx 3.2$. It is known (cf.~\cite{ElderFusyRechnitzerCountingGeodesics}, \cite{GubaCayLeyGraphF}) that the size of the spheres $S_n$  of radius $n$ in the word metric of $F$, grows roughly with a factor $\phi^2=2.618\ldots$ (the square of the golden ratio $\phi=\tfrac{1+\sqrt 5}2$), when passing from $S_n$ to $S_{n+1}$.
Hence we expect the size of $Y_n\subset Y^n=(S_1)^n=S_n\sqcup S_{n-2}\sqcup S_{n-4}\sqcup\ldots$ to grow with a factor close to $2.618$ when passing from $n$ to $n+1$.
By using the recursion formula
$$h_{n+1}=h\cdot h_n-q h_{n-1}$$
 derived in Section~\ref{s:five}(see formulas~(\ref{e:hnplus1knplus1Even})-(\ref{e:hnplus1knplus1Odd}) for the case $h_n=k_n$)
we could compute the list of multiplicities $(c_x^{n+1})_{x\in Y_{n+1}}$ from the lists $(c_x^{n})_{x\in Y_n}$ and $(c_x^{n-1})_{x\in Y_{n-1}}$, and by this method the computation time only grew $\approx 2.8$ when passing from $n$ to $n+1$. This worked well for computing the numbers $||h_n||_2^2$ for $n\le 23$. To obtain $||h_{24}||_2^2$  an extra trick was implemented
to speed up the computations and to reduce the size of the storing data.
The elements of $Y_n\subset \Gamma$ are homeomorphisms written as forest diagrams $(u,v)$ which consist of two ``reduced" trees -- the domain tree $u$, and the range tree $v$.
The extra trick consists on the observation that if $(u,v)\in Y_n\backslash\{e\}$ then so is its inverse $(v,u)\in Y_n\backslash\{e\}$.
Saving only one of them reduces the storing size by about one half.
This program was run in a standard desktop computer with an AMD Phenom II X4 965 Quad-core processor (3400 Mhz) and 2 TB of SSD-hard disk.
It took about 5 days to compute and store the forest diagrams (in serialized form) for $||h_{24}||_2^2$ for $Y=\{A,B,A^{-1},B^{-1}\}$, and another five days to do the sorting.
The series of numbers $||h_n||_2^2$, $\eta_n$, $\zeta_n$, $m_n(h^*h)$ for case 1 (resp. case 2) are shown in Table~\ref{t:Case1SeriesOfNumbers} (resp.~Table ~\ref{t:Case2SeriesOfNumbers}).

In Appendix \ref{a:1} (Theorem~\ref{t:B1}($iii$) ) we show that if $\Gamma$ is torsion free (e.g. if $\Gamma=F$), then the numbers
$$
 \zeta'_n:=\sum\limits_{\substack{d|n}} \mu(\frac nd)\zeta_d, \qquad n\in\N
$$
are non negative integers divisible by 2n.
Here $\mu:\N\to\{-1,0,1\}$  is the M\"obius function (cf.~Section 2.2. in \cite{Apostol}).
We used this as a test to detect possible programming errors in the computation of the sequence of numbers $||h_n||_2^2$.
This test rules out false values of $||h_n||_2^2$ with probability $1-\frac 1{2n}$.

\section{\textbf{Estimating Norms}}\label{s:four}

Let $T\in L(\Gamma)$ be an operator in the von Neumann algebra of a countable group $\Gamma$.
We will in this section discuss how we can get good lower bounds for $||T||$ from knowing only finitely many of the moments
$m_n(T^*T)=\tau((T^*T)^n)$ as well as getting some prediction of what the actual values of $||T||$ might be. The following proposition is well known. For the convenience of the reader, we include a proof.

\begin{pro}\label{p:4.1}

  Let $T\in L(\Gamma)$, $T\ne0$. Then
  \begin{equation}\label{e:4A}
    m_n(T^*T)^{1/n} \le \frac{m_n(T^*T)}{m_{n-1}(T^*T)}\le ||T||^2\qquad n\in\N.
  \end{equation}
  Moreover, both $(m_n^{1/n})_{n=1}^\infty$ and $(\frac{m_n}{m_{n-1}})_{n=1}^\infty$ are increasing sequences and
  \begin{equation}\label{e:4B}
    \lim\limits_{n\to\infty} m_n(T^*T)^{1/n}=\lim\limits_{n\to\infty} \frac{m_n(T^*T)}{m_{n-1}(T^*T)}= ||T||^2.
  \end{equation}
\end{pro}

\begin{proof}
Let $\nu=\mu_{T^*T}$ be the spectral distribution measure of $T^*T$ with respect to the trace $\tau$ on $L(\Gamma)$.
Since $T^*T\ge0$ is a positive self-adjoint operator, we have from Section~\ref{s:two} that
$$\Supp(\nu)=\sigma(T^*T)\subset [0,\infty)$$
  and
$$\max(\Supp(\nu))=\max\{|x|\mid x\in\supp(\nu)\}=||T^*T||=||T||^2.$$
Let $m_n:=m_n(T^*T)$, $n\in\N_0$. Then
$$m_n=\int_{0}^{||T||^2} t^n \,d_\nu(t)>0$$
because $\Supp(\nu)\ne\{0\}$ as $T\ne0$.
Using $t^n\le ||T||^2t ^{n-1}$ for $t\in[0,||T||^2]$, we get $m_n/m_{n-1}\le ||T||^2$ for $n\ge 1$. Moreover by H\" older's inequality, we have
\begin{eqnarray*}
  m_n&=& \int_{0}^{||T||^2} (t^{n-1})^{1/2}(t^{n+1})^{1/2}\,d\nu(t)\\
     &\le& \left(\int_{0}^{||T||^2} t^{n-1}\,d\nu(t)\right)^{1/2}\left(\int_{0}^{||T||^2} t^{n+1}\,d\nu(t)\right)^{1/2}\\
     &\le&m_{n-1}^{1/2}\cdot m_{n+1}^{1/2}.
\end{eqnarray*}
Hence, $m_n/m_{n-1}\le m_{n+1}/m_{n}$, i.e. $(m_n/m_{n-1})_{n=1}^\infty$ is an increasing sequence. Since $m_0=1$,
$$m_n=\frac{m_1}{m_0}\frac{m_2}{m_1}\cdots\frac{m_n}{m_{n-1}}\le (\frac{m_n}{m_{n-1}})^n$$
which proves (\ref{e:4A}).
Using H\" older's inequality we have
$$m_n=\int_0^{||T||^2}t^n\cdot1\,d\nu(t)\le \left(\int_0^{||T||^2}t^{n+1}\,d\nu(t)\right)^{n/(n+1)}\left(\int_0^{||t||^2}1\,d\nu(t)\right)^{1/(n+1)}=m_{n+1}^{n/(n+1)}$$
which proves that $(m_n^{1/n})_{n=1}^\infty$ is an increasing sequence.
Let $0<a<||T||^2$. Since $||T||^2\in\Supp(\nu)$,
$$\nu([a,||T||^2])>0.$$
Since $$m_n\ge \int_a^{||T||^2} t^n\,d\nu(t)\ge a^n \nu([0,||T||^2]).$$
It follows that $\liminf_{n\to\infty}(m_n^{1/n})\ge a$ for all $a\in(0,||T||^2)$. Hence $\liminf_{n\to\infty}(m_n^{1/n})\ge ||T||^2$ which together with (\ref{e:4A}) proves (\ref{e:4B}).
\end{proof}

The lower bounds $m_n(T^*T)^{\frac1{2n}}$ and $\frac{m_n(T^*T)^{1/2}}{m_{n-1}(T^*T)^{1/2}}$ for $||T||$ in Proposition~\ref{p:4.1} give, however, rather poor lower estimates of $||T||$ in the two main cases we consider, namely
$$T_1=I+A+B,\qquad T_2=A+A^{-1}+B+B^{-1}$$
in the group ring $\C F$ of the Thompson group $F$.
To get better estimates, we apply methods from the theory of orthogonal polynomials (c.f \cite{S} or \cite{HTF2}) to the symmetric measure $\mu_{\tilde T}$ described in Section~\ref{s:two}.

Let $\mu$ be a compactly supported Radon measure on $\R$ for which $\supp(\mu)$ is not a finite set. Then the polynomials $1$, $t$, $t^2$, $\ldots$ are linear independent in $L^2(\mu)$. Hence by Gram-Schmidt orthonormalization we get a sequence of polynomials $(p_n)_{n=0}^\infty$  of degree$(p_n)=n$, such that
$$\int_{\R} p_m(t)p_n(t)\,d\mu(t)=\delta_{m,n}, \qquad m,n\in\N_0.$$

Moreover, the polynomials $(p_n)_{n=0}^\infty$ are uniquely determined if we add the condition that $k_n>0$, where $k_n$ is the coefficient of $t^n$ in the polynomial $p_n(t)$. Let
$$c_n:=\int_{\R} t^n \,d\mu(t),\qquad n\in\N_0$$
be the n'th moment of $\mu$, and put
\begin{equation}\label{e:4.1}
  D_n:=\det ([c_{i+j}]_{i,j=0}^n).
\end{equation}
Then by Formula (2.2.7) and (2.2.15) in \cite{S} we have
\begin{eqnarray}
\label{e:4.2}  &&D_n>0 \qquad\qquad\qquad\quad  n\ge0\\
\label{e:4.3}  &&k_n=\left(\frac{D_{n-1}}{D_n}\right)^{1/2}\qquad n\ge 1.
\end{eqnarray}
Note that $k_0=c_0^{-1/2}=D_0^{-1/2}$. Hence (\ref{e:4.3}) holds for all $n\ge0$ if we define
\begin{equation}\label{e:4.4}
  D_{-1}:=1.
\end{equation}
Note also, that since $\mu$ is compactly supported, the set of polynomials is dense in $L^2(\mu)$. Hence $(p_n)_{n=0}^\infty$ is an orthonormal basis for the Hilbert space $L^2(\mu)$.

If we furthermore assume that $\mu$ is a symmetric measure on $\R$ (i.e. $\mu(-B)=\mu(B)$ for every Borel set $B$), then it is easily seen that
\begin{equation}\label{e:4.5}
  p_n(-t)=(-1)^np_n(t),\qquad n\in\N_0.
\end{equation}
An elementary application of the recursion formula for orthonormal polynomials (c.f Section 3.2 in \cite{S}) now gives:

\begin{pro}\label{p:4.2}

  Let $\mu$ be a symmetric compactly supported Radon measure on $\R$ for which the support is not finite. Then the multiplication operator
  $$m_t:f(t)\to t f(t),\qquad f\in L^2(\mu)$$
  in $B(L^2(\mu))$ has the form
  \begin{equation}\label{e:4.6}
    M=[m_{ij}]_{i,j=0}^\infty=\left(
     \begin{array}{ccccc}
       0 & \alpha_1 & 0 & 0 &\cdots \\
       \alpha_1 & 0 & \alpha_2 & 0 &\cdots \\
       0 &  \alpha_2 &  0 & \alpha_3 &\cdots\\
       0 & 0 &  \alpha_3 &  0 &\cdots\\
       \vdots & \vdots & \vdots &\vdots & \ddots \\
     \end{array}
   \right)
  \end{equation}
  with respect to the orthonormal basis $(p_n)_{n=0}^\infty$ for $L^2(\mu)$, where
  \begin{equation}\label{e:4.7}
    \alpha_n=\frac{k_{n-1}}{k_n}=\left(\frac{D_{n-2}D_n}{D_{n-1}^2}\right),\qquad n\ge 1.
  \end{equation}
\end{pro}
\begin{proof}
  From Theorem~3.2.1 in \cite{S} or pp.~58-59 in \cite{HTF2} we have
  $$p_{n+1}(t)=(A_n t+B_n)p_n(t)-C_n p_{n-1} (t),\qquad n\ge 1$$
  and
  $$p_1(t)=(A_0 t+ B_0) p_0(t),$$
  where $A_n=\frac{k_{n+1}}{k_n}$ $(n\ge 0)$ and $C_n=\frac{A_n}{A_{n-1}}$ $(n\ge 1)$.
  Moreover, $B_n=0$ for all $n\ge0$ by (\ref{e:4.5}). Hence
  \begin{eqnarray*}
    tp_n(t)&=&\frac{1}{A_n} p_{n+1}(t)+\frac{C_n}{A_n}p_{n-1}(t)\\
    &=&\frac{k_n}{k_{n+1}}p_{n+1}(t)+\frac{k_{n-1}}{k_n}p_{n-1}(t),\qquad n\ge1\\
  \end{eqnarray*}
  and
  $$tp_0(t)=\frac{k_0}{k_1} p_1(t).$$
 This shows that the matrix for $m_t$ with respect to the basis $(p_n)_{n=0}^\infty$ has the form (\ref{e:4.6}) where $\alpha_n=\frac{k_{n-1}}{k_n}$ $(n\ge 1)$. The second formula for $\alpha_n$ in (\ref{e:4.7}) follows from (\ref{e:4.3}) and (\ref{e:4.4}).
\end{proof}

\begin{pro}\label{p:4.3}

  Let $\mu$ and $M$ be as in Proposition~\ref{p:4.2}.
  Let $$M_n=[m_{ij}]_{i,j=0}^n=\left(
     \begin{array}{ccccc}
              0 & \alpha_1 &         &  & \\
       \alpha_1 & 0        & \alpha_2&  &0 \\
                & \alpha_2 & \ddots  & \ddots &\\
                &          &  \ddots & \ddots  &\alpha_n\\
       0        &          &         &\alpha_n & 0 \\
     \end{array}
   \right)$$
   and let $||M||$ and $||M_n||$ denote the norm of $M$ (resp. $M_n$) considered as operators on $\ell^2(\N_0)$ (resp. $\ell^2(\{0,\ldots,n\})$). Then
   \begin{eqnarray}
     \label{e:4.8} &&||M||=\max(\Supp(\mu))\\
     \label{e:4.9}&&||M_n||=\lambda_{\max}(M_n),
   \end{eqnarray}
   where $\lambda_{\max}$  denotes the largest eigenvalue of the symmetric matrix $M_n$. Moreover,
   \begin{eqnarray}
     \label{e:4.10}&& (||M_n||)_{n=1}^\infty \text{ is an increasing sequence.}\\
     \label{e:4.11}&& \lim\limits_{n\to\infty} ||M_n||=||M||.
   \end{eqnarray}
\end{pro}
\begin{proof}
  Since $(p_n)_{n=0}^\infty$ is an orthonormal basis for $L^2(\mu)$
  $$m_t=UMU^*$$
  where $U$ is the unitary matrix from $\ell^2(\N_0)$ to $L^2(\mu)$ for which $U\delta_n=p_n$, where $(\delta_n)_{n=0}^\infty$ is the standard basis for $\ell^2(\N_0)$. Hence
  $$||M||=||m_t||_{B(L^2(\mu))}.$$
  But the norm of the multiplication operator $m_t$ is just the $L^\infty(\mu)$-norm of the function $t\to t$ $(t\in\R)$. Hence
  $$||M||=\max\{|x|\mid x\in \Supp(\mu)\}=\max(\Supp(\mu)),$$
  where the last equality follows from the assumption that $\mu$ is symmetric. This proves (\ref{e:4.8}).
  Since $M_n$ is a self-adjoint  matrix
  $$||M_n||=\max\{|\lambda|\mid \lambda \text{ eigenvalue of } M_n \}.$$
  Since $G_n M_n G_n^{-1}= -M_n$, where $G_n$ is the diagonal matrix:
  $$G_n=\mathrm{diag}(1,-1,1,-1,\ldots,(-1)^n),$$
  the set of eigenvalues of $M_n$ and $-M_n$ are the same. Hence
  $$||M_n||=\lambda_{\max}(M_n)=-\lambda_{min}(M_n)$$
  proving (\ref{e:4.9}). Finally (\ref{e:4.10}) and (\ref{e:4.11}) follows from the fact that $\ell^2\{(0,\ldots,n)\}$ is an increasing sequence of subspaces of $\ell^2(\N_0)$ whose union is dense in $\ell^2(\N_0)$.
\end{proof}

\begin{pro}\label{p:4.4}

  Let $\mu$ and $M$ be as in Proposition~\ref{p:4.2}. Then
  \begin{equation}\label{e:4.12and4.13}
    \liminf_{n\to\infty}(\alpha_{n-1}+\alpha_n)\quad \le \quad ||M||\quad \le\quad  \sup_{n\ge 2}(\alpha_{n-1}+\alpha_n).
  \end{equation}
  If $\frac 1{\sqrt2}\alpha_1\le \alpha_2\le \alpha_3$ then
  \begin{equation}\label{e:4.14}
    ||M||\quad \le\quad  \sup_{n\ge 3}(\alpha_{n-1}+\alpha_n).
  \end{equation}
\end{pro}

\begin{proof}
  Let
  $$a=\liminf_{n\to\infty}(\alpha_{n-1}+\alpha_n)=\liminf_{n\to\infty}(\alpha_n+\alpha_{n+1}).$$
  Let $\varepsilon>0$ and choose $n_0\in\N$ such that $\alpha_n+\alpha_{n+1}\ge a-\varepsilon$ for $n\ge n_0$. For $k\in\N$ consider the unit vector $x=(x_i)_{i=0}^\infty\in\ell^2(\N_0)$ given by
  $$x_i=\left\{
          \begin{array}{ll}
            (2k+1)^{-1/2}, & n_0\le i \le n_0+2k \\
            0, & \text{otherwise.}
          \end{array}
        \right.$$
Let $\langle \cdot,\cdot \rangle$ denote the inner product in $\ell^2(\N_0)$. Then
\begin{eqnarray*}
  ||M||\ge \langle M x, x\rangle=\frac 2 {2k+1} (\alpha_{n_0}+\alpha_{n_0+1}+\cdots+\alpha_{n_0+2k-1})\ge \frac{2k}{2k+1}(a-\varepsilon).
\end{eqnarray*}
Letting first $k\to\infty$ and next $\varepsilon\to 0$ we get $||M||\ge a$ proving the first inequality in (\ref{e:4.12and4.13}).
Since $M=[m_{ij}]_{i,j=0}^\infty$ is a symmetric matrix with non-negative entries, it follows from Schur's test (see e.g. Exercise 3.2.17 in \cite{GKP}) that if there exists a sequence $(q_n)_{n=0}^\infty$ of strictly positive real numbers and a constant $C>0$ such that
$$\sum_{j=0}^\infty m_{ij}q_j\le C q_i,\qquad  i\in\N_0$$
then $||M||\le C$. Applying this to the constant sequence $q_n=1$, $n\in\N_0$, we get
$$||M||\le \sup\{\alpha_1,\alpha_1+\alpha_2,\alpha_2+\alpha_3,\ldots\}$$
proving the second inequality in (\ref{e:4.12and4.13}).
If we instead let $q_0=\frac 1 {\sqrt 2}$ and $q_n=1$ for all $n\ge 1$, we get
$$||M||\le \sup\{\sqrt 2\alpha_1,\frac 1{\sqrt 2} \alpha_1+\alpha_2, \alpha_2+\alpha_3,\alpha_3+\alpha_4,\ldots\}.$$
Under the assumptions of (\ref{e:4.14}),
$$\max\{\sqrt 2 \alpha_1, \frac 1 {\sqrt 2} \alpha_1+\alpha_2\} \le 2\alpha_2\le \alpha_2+\alpha_3.$$
Hence  (\ref{e:4.14}) holds.
\end{proof}

\begin{rem}\label{r:4.5}

  Note that for $T\in L(\Gamma)$ we have from Section~\ref{s:two}, that
$$m_n(T^*T)=\int_{-||T||}^{||T||} t^{2n} d\mu_{\tilde T}(t)\, dt,\qquad n\in\N_0,$$
and $\mu_{\tilde T}$  is a symmetric compactly supported probability measure on $\R$ and
$$||T||=\max(\Supp(\mu_{\tilde T})).$$
Hence, when applying Propositions~\ref{p:4.2}, \ref{p:4.3} and \ref{p:4.4} to $\mu=\mu_{\tilde T}$ we have
\begin{eqnarray}
&&c_{2n}=m_n(T^*T)\quad\text{and}\quad c_{2n+1}=0,\qquad n\in\N_0\\
&&||T||=||M||=\lim_{n\to\infty} ||M_n||
\end{eqnarray}
and if $\frac 1 {\sqrt 2} \alpha_1\le \alpha_2 \le \alpha_3$,
\begin{equation}\label{e:liminflimsup}
  \liminf_{n\to\infty}(\alpha_{n-1}+\alpha_n) \le ||T|| \le \sup_{n\ge 3} (\alpha_{n-1}+\alpha_n).
\end{equation}
The condition $\frac 1 {\sqrt 2} \alpha_1\le \alpha_2 \le \alpha_3$ is fulfilled in the two main cases that we consider, namely $T_1=I+A+B$, and  $T_2=A+A^{-1}+B+B^{-1}$ (see Tables~\ref{t:Case1NumbersForEstimatingTheNorm}-\ref{t:Case2NumbersForEstimatingTheNorm}).
Actually in these two cases, the right hand side of (\ref{e:liminflimsup}) can be sharpened further to
\begin{equation}
   ||T|| \le \limsup_{n\to\infty} (\alpha_{n-1}+\alpha_n)
\end{equation}
  (See Appendix \ref{a:3}, Corollary~\ref{c:A8}).
\end{rem}

\begin{rem}\label{r:4.6}

  If $\mu$ is finitely supported, then $L^2(\R,\mu)$ is finite dimensional and
$$d=\dim L^2(\R,\mu)=|\Supp(\mu)|.$$
Therefore, the Gram-Schmidt orthonormalization process will end after $d-1$ steps and we get a finite matrix
$$M=\left(
     \begin{array}{ccccc}
              0 & \alpha_1 &         &  & \\
       \alpha_1 & 0        & \alpha_2&  &0 \\
                & \alpha_2 & \ddots  & \ddots &\\
                &          &  \ddots & \ddots  &\alpha_{d-1}\\
       0        &          &         &\alpha_{d-1} & 0 \\
     \end{array}
   \right)
$$
where the numbers $\alpha_n$ $(1\le n\le d-1)$ are given by (\ref{e:4.7}).
This will however not happen in the two main cases we are interested in, (see Appendix \ref{a:3}, Corollary~\ref{c:A4}).
\end{rem}

The following proposition extends Proposition~\ref{p:4.1}.

\begin{pro}\label{p:4.7}

Let $T\in L(\Gamma)$, $T\ne0$, and let $(M_n)_{n=1}^\infty$ be as in Remark~\ref{r:4.5}. Then
$$\frac{m_n(T^*T)}{m_{n-1}(T^*T)}\le ||M_n||^2=(\lambda_{\max}(M_n))^2\le ||T||^2,\qquad n\in\N.  $$
\end{pro}

 Let $\delta_0, \delta_1,\delta_2,\ldots$ denote the standard orthonormal basis for   $\ell^2(\N_0)$ (i.e. $\delta_0=(1,0,0,\ldots)$), and
$\delta'_0, \delta'_1,\delta'_2,\ldots,\delta'_n$ denote the standard basis for $\ell^2(\{0,1,\ldots,n\})$.
For the proof of Proposition~\ref{p:4.7} we will need the following lemma:

\begin{lem} \label{l:mnTOMn}

  For $n\in\N_0$ and $k=0,\ldots,n$, the moments $m_k(T^*T)$ are given by $$m_k(T^*T)=||(M_n)^k\delta'_0||^2_2.$$
\end{lem}
\begin{proof}
 Recall that $p_0=1$ as it is the first element of the orthonormal basis (i.e. first step in Gram-Schmidt).
 Since $p_0,\ldots,p_n$ is an orthonormalization of $1,t,\ldots,t^n$, we have
 $$(m_t)^np_0=(m_t)^n1=t^n\in\Span(1,t,\ldots,t^n)=\Span(p_0,p_1,\ldots,p_n).$$
 Thus for $k=0,\ldots,n$ we have
 $$M^k\delta_0\in\Span(\delta_0,\ldots,\delta_n)$$
  and
 $$||M^k\delta_0||_2^2=||(m_t)^k p_0||_2^2=||t^k||_2^2=\int_{-||T||}^{||T||} t^{2k}\,d\mu_{\tilde T}(t)=m_k(T^*T),$$
 as $M$ is the matrix of $m_t$ with respect to the ONB $p_0,p_1,\ldots$.
 Let $E_n$ be the projection of $\ell^2(\N_0)$ onto $\Span(\delta_0,\ldots,\delta_n) $.
 We claim that
  $$(E_nME_n)^k\delta_0=M^k\delta_0.$$
 Since
 $$E_nME_n=\left(
                  \begin{array}{cc}
                    M_n & 0 \\
                    0 & 0 \\
                  \end{array}
                \right),$$
it follows from the claim that $$||(M_n)^k\delta'_0||_2^2=||(E_nME_n)^k\delta_0||_2^2=||M^k\delta_0||_2^2=m_k(T^*T),$$
which proves the Lemma.
The proof of  the claim is done by induction in $k$. For $k=1$ we have
$E_nM E_n\delta_0=E_nM\delta_0=M\delta_0$ because $M\delta_0\in\Span\{\delta_0,\delta_1\}$ and $n\ge 1$.
Assume next that $(E_nME_n)^k\delta_0=M^k\delta_0$ for  $k\le n-1$. Then
$(E_nME_n)^{k+1}\delta_0=E_nME_nM^k\delta_0=E_nM^{k+1}\delta_0=M^{k+1}\delta_0$ as $M^k\delta_0\in\Span(\delta_0,\ldots,\delta_n)$, which proves the claim.

\end{proof}

\begin{proof}[Proof of Proposition~\ref{p:4.7}]
By the previous Lemma~\ref{l:mnTOMn} we have
\begin{eqnarray*}
   m_n=||(M_n)^n\delta'_0||_2^2=||M_n(M_n)^{n-1}\delta'_0||_2^2\le||M_n||^2||(M_n)^{n-1}\delta'_0||_2^2=||M_n||^2m_{n-1}
 \end{eqnarray*}
 Hence the first inequality of the proposition.
The rest of the proof of Proposition~\ref{p:4.7} follows from Proposition~\ref{p:4.3} and Remark~\ref{r:4.5}.
\end{proof}
\begin{table}[b]
\caption{ Estimating the norm $||I+A+B||$.\label{t:Case1NumbersForEstimatingTheNorm}}
{\footnotesize
\begin{tabular}{cccccc}
\hline
$n$ & $m_n^{\frac1{2n}}$ & $\left(\frac{m_n}{m_{n-1}}\right)^{\frac12}$ & $\alpha_n$ & $\lambda_{\max}(M_n)$& $\alpha_{n-1}+\alpha_n$\\
\hline
1 & 1.73205 & 1.73205 & 1.73205 & 1.73205 & $---$ \\
2 & 1.96799 & 2.23607 & 1.41421 & 2.23607 & 3.14626 \\
3 & 2.10501 & 2.40832 & 1.41421 & 2.44949 & 2.82843 \\
4 & 2.19710 & 2.49828 & 1.41421 & 2.56155 & 2.82843 \\
5 & 2.26432 & 2.55438 & 1.41421 & 2.62860 & 2.82843 \\
6 & 2.31606 & 2.59306 & 1.41421 & 2.67233 & 2.82843 \\
7 & 2.35741 & 2.62151 & 1.41421 & 2.70265 & 2.82843 \\
8 & 2.39140 & 2.64342 & 1.44338 & 2.72620 & 2.85759 \\
9 & 2.41994 & 2.66090 & 1.41303 & 2.74445 & 2.85641 \\
10 & 2.44434 & 2.67524 & 1.43768 & 2.75941 & 2.85071 \\
11 & 2.46548 & 2.68728 & 1.41733 & 2.77154 & 2.85500 \\
12 & 2.48403 & 2.69756 & 1.44610 & 2.78200 & 2.86343 \\
13 & 2.50048 & 2.70649 & 1.41406 & 2.79072 & 2.86016 \\
14 & 2.51518 & 2.71434 & 1.45286 & 2.79850 & 2.86691 \\
15 & 2.52842 & 2.72131 & 1.41509 & 2.80515 & 2.86795 \\
16 & 2.54043 & 2.72757 & 1.45239 & 2.81118 & 2.86749 \\
17 & 2.55138 & 2.73323 & 1.41982 & 2.81645 & 2.87222 \\
18 & 2.56143 & 2.73839 & 1.45400 & 2.82132 & 2.87382 \\
19 & 2.57069 & 2.74311 & 1.42044 & 2.82563 & 2.87444 \\
20 & 2.57925 & 2.74746 & 1.45571 & 2.82966 & 2.87615 \\
21 & 2.58720 & 2.75148 & 1.42133 & 2.83326 & 2.87704 \\
22 & 2.59461 & 2.75522 & 1.45768 & 2.83665 & 2.87901 \\
23 & 2.60154 & 2.75871 & 1.42269 & 2.83972 & 2.88037 \\
24 & 2.60803 & 2.76198 & 1.45807 & 2.84263 & 2.88076 \\
25 & 2.61414 & 2.76505 & 1.42638 & 2.84529 & 2.88445 \\
26 & 2.61989 & 2.76793 & 1.45596 & 2.84782 & 2.88234 \\
27 & 2.62533 & 2.77066 & 1.42841 & 2.85014 & 2.88437 \\
28 & 2.63047 & 2.77323 & 1.45785 & 2.85237 & 2.88626 \\
29 & 2.63535 & 2.77568 & 1.42883 & 2.85443 & 2.88669 \\
30 & 2.63998 & 2.77800 & 1.45815 & 2.85641 & 2.88698 \\
31 & 2.64439 & 2.78021 & 1.43056 & 2.85824 & 2.88871 \\
32 & 2.64860 & 2.78231 & 1.45854 & 2.86002 & 2.88910 \\
33 & 2.65261 & 2.78432 & 1.43178 & 2.86167 & 2.89032 \\
34 & 2.65645 & 2.78625 & 1.45860 & 2.86327 & 2.89039 \\
35 & 2.66012 & 2.78809 & 1.43344 & 2.86477 & 2.89204 \\
36 & 2.66365 & 2.78986 & 1.45806 & 2.86623 & 2.89150 \\
37 & 2.66702 & 2.79155 & 1.43523 & 2.86759 & 2.89329 \\
\hline
\end{tabular}}
\end{table}

\begin{table}[h]
\caption{Estimating the norm  $||A+B+A^{-1}+B^{-1}||$.\label{t:Case2NumbersForEstimatingTheNorm}}
{\footnotesize
\begin{tabular}{cccccc}
\hline
$n$ & $m_n^{\frac1{2n}}$ & $\left(\frac{m_n}{m_{n-1}}\right)^{\frac12}$ & $\alpha_n$ & $\lambda_{\max}(M_n)$& $\alpha_{n-1}+\alpha_n$\\
\hline
1 & 2.00000 & 2.00000 & 2.00000 & 2.00000 & $---$\\
2 & 2.30033 & 2.64575 & 1.73205 & 2.64575 & 3.73205\\
3 & 2.47884 & 2.87849 & 1.73205 & 2.93352 & 3.46410\\
4 & 2.60058 & 3.00287 & 1.73205 & 3.08891 & 3.46410\\
5 & 2.69061 & 3.08298 & 1.78471 & 3.19184 & 3.51676\\
6 & 2.76083 & 3.14038 & 1.76554 & 3.26439 & 3.55025\\
7 & 2.81769 & 3.18437 & 1.79564 & 3.32000 & 3.56119\\
8 & 2.86506 & 3.21963 & 1.77500 & 3.36276 & 3.57064\\
9 & 2.90535 & 3.24883 & 1.81110 & 3.39790 & 3.58610\\
10 & 2.94023 & 3.27358 & 1.79214 & 3.42682 & 3.60324\\
11 & 2.97083 & 3.29495 & 1.81693 & 3.45164 & 3.60907\\
12 & 2.99800 & 3.31368 & 1.80006 & 3.47272 & 3.61699\\
13 & 3.02234 & 3.33028 & 1.82585 & 3.49133 & 3.62591\\
14 & 3.04432 & 3.34515 & 1.80841 & 3.50755 & 3.63425\\
15 & 3.06433 & 3.35858 & 1.83203 & 3.52217 & 3.64043\\
16 & 3.08264 & 3.37080 & 1.81426 & 3.53513 & 3.64629\\
17 & 3.09949 & 3.38198 & 1.83702 & 3.54697 & 3.65129\\
18 & 3.11507 & 3.39228 & 1.82036 & 3.55761 & 3.65739\\
19 & 3.12954 & 3.40180 & 1.84173 & 3.56745 & 3.66210\\
20 & 3.14303 & 3.41065 & 1.82478 & 3.57639 & 3.66651\\
21 & 3.15565 & 3.41890 & 1.84556 & 3.58471 & 3.67034\\
22 & 3.16749 & 3.42663 & 1.82942 & 3.59235 & 3.67498\\
23 & 3.17863 & 3.43388 & 1.84878 & 3.59952 & 3.67820\\
24 & 3.18914 & 3.44071 & 1.83311 & 3.60613 & 3.68189\\
\hline
\end{tabular}}
\end{table}

\begin{figure}[H]
\centerline{
\includegraphics[scale=.65]{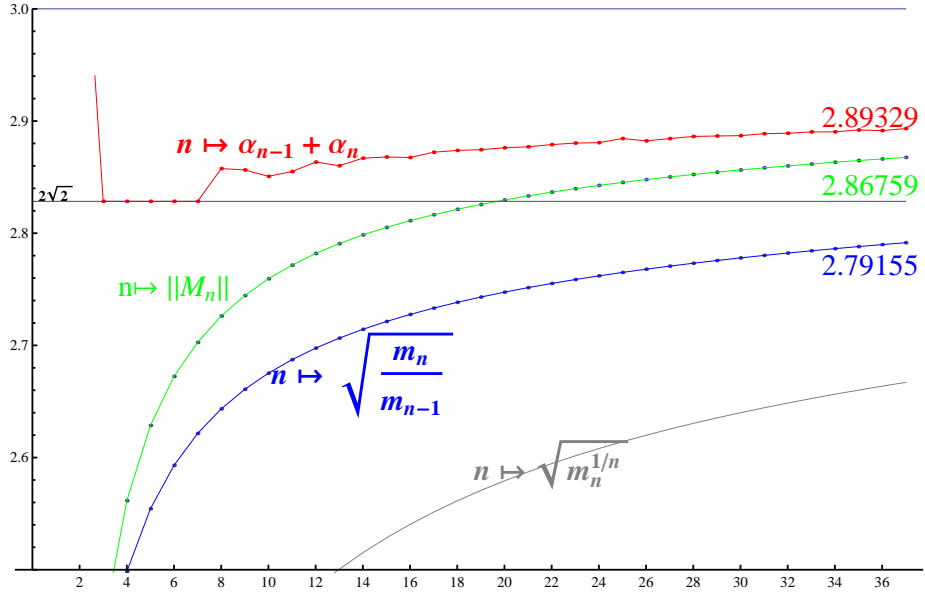}
}
\vspace*{8pt}
  \caption{Estimating the norm $||I+A+B||$.\label{f:IABgraphs}}
\end{figure}

\begin{figure}[h]
\centerline{\includegraphics[scale=.7]{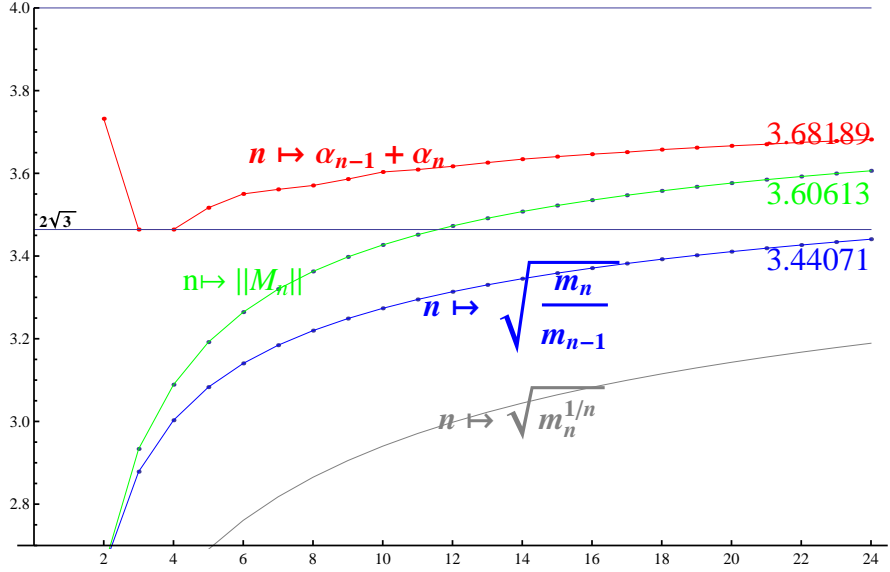}}
\vspace*{8pt}
  \caption{Estimating the norm $||A+B+A^{-1}+B^{-1}||$\label{f:ABAinvBinvgraphs}}
\end{figure}

Recall that from Proposition~\ref{p:4.1} and Proposition~\ref{p:4.7} that
$$m_n^{\frac{1}{2n}} \le \left(\frac{ m_n}{m_{n-1}}\right)^{1/2}\le \lambda_{\max}(M_n)\le ||T||,$$
and all three sequences on the left hand side converges monotonically to $||T||$ as $n\to \infty$.
Note that in both cases $m_n^{\frac 1{2n}}$ and $(m_n/m_{n-1})^{1/2}$ are poor lower estimates
(see Tables~\ref{t:Case1NumbersForEstimatingTheNorm}, \ref{t:Case2NumbersForEstimatingTheNorm} and Figs.~\ref{f:IABgraphs}, \ref{f:ABAinvBinvgraphs}).
For the listed range of integers $n$, they both stay well below the known lower bound $||I+A+B||>2\sqrt 2\approx 2.82824$ in case 1, (resp. $||A+A^{-1}+B+B^{-1}||>2\sqrt 3\approx 3.46410$ in case 2),
while the lower estimates $\lambda_{\max}(M_n)$ stay above this value for $n\ge 20$ in case 1 (resp. for $n\ge 12$ in case 2).
The best exact lower bound for $||T||$ in case 1 ($T=I+A+B$) we can obtain from our results is
$$||I+A+B||\ge \lambda_{\max}(M_{37})= 2.86759.$$
Note however that by Proposition~\ref{p:4.4},
$$||I+A+B||\ge \liminf_{n\to\infty} (\alpha_{n-1}+\alpha_n)$$
and since $(\alpha_{n-1}+\alpha_n)_{n=1}^\infty$ appear to be monotonically increasing for $n\ge 26$ our computation results make it very likely that actually
$$||I+A+B||\ge \alpha_{36}+\alpha_{37}=2.89329.$$
To get some prediction of the actual value of $||I+A+B||$, we made a least squares fitting of the 26 numbers
$$\lambda_{\max}(M_n),\qquad {12\le n\le 37}$$
to a function of the following form
$$f(n)=a-b(n-c)^{-d},\qquad n=12,\ldots,37$$
and found that the optimal values of the parameters $a,b,c,d$ were
$$a=2.950\qquad b=0.630\qquad c=0.571\qquad d=1.900.$$
In particular, this extrapolation argument predicts that
$$||I+A+B||=\lim_{n\to\infty} \lambda_{\max}(M_n)\approx 2.950.$$
However, we can in no way rule out that $||I+A+B||=3$, i.e. that $F$ is amenable.
In the same way, we get in Case 2, ($T=A+B+A^{-1}+B^{-1}$) the precise lower bound
$$||A+A^{-1}+B+B^{-1}||\ge \lambda_{\max}(M_{24})= 3.60613$$
and that most likely we have
$$||A+A^{-1}+B+B^{-1}||\ge \alpha_{23}+\alpha_{24}= 3.68189.$$
Moreover, by making a least squares fitting of the $17$ numbers
$$\lambda_{\max}(M_n),\qquad{8\le n\le 24}$$
to a function of the form
$$f(n)=a-b(n-c)^{-d},\qquad n=8,\ldots,24$$
we found the values of $a,b,c,d$ to be
$$a=3.870\qquad b=1.612 \qquad c=0.573 \qquad d=0.480.$$
In particular, this extrapolation method predicts that
$$||A+A^{-1}+B+B^{-1}||=\lim_{n\to\infty} \lambda_{\max}(M_n)\approx 3.870$$
but again we cannot rule out that $F$ is amenable.

In \cite{BurilloClearyWiest}, Burillo, Cleary and Wiest used probabilistic methods to estimate the moments $m_n$ in Case 2 ($T=A+A^{-1}+B+B^{-1}$) for $n=10,20,\ldots,160$. In their notation, $L=2n$ and $m_n=4^L\hat p(L)$. They found that (see Table~1 in \cite{BurilloClearyWiest})
$$m_{160}^{1/320}=4\hat p(320)^{1/320}\approx 4\cdot 0.9003=3.6012$$
and
$$\left(\frac{ m_{160}}{m_{150}}\right)^{1/20}=4\left(\frac{\hat p(320)}{\hat p(300)}\right)^{1/20}\approx 4\cdot 0.9161=3.6644.$$
Since $m_n/m_{n-1}$  is increasing slowly for $151\le n\le 160$ we also have
$$\left(\frac{m_{160}}{m_{159}}\right)^{1/2}\approx \left(\frac{m_{160}}{m_{150}}\right)^{1/20}\approx 3.6644.$$
Their estimates are based on random samples of words in the generators, and are therefore not precise lower bound for $||A+A^{-1}+B+B^{-1}||$.
In comparison, we found an exact lower bound $3.60613$ of the norm based only on $(m_n)_{1\le n\le 24}$ and a very likely lower bound $3.68189$ based on the same list of moments.
In \cite{ElderRechnitzerWongCogrowth}, Elder, Rechnitzer and Wong also worked on estimating the norm $||A+A^{-1}+B+B^{-1}||$.
They found the lower bound
$$||A+A^{-1}+B+B^{-1}||\ge 3.55368$$
based on computing numerically the largest eigenvalue of the adjacency matrix for the Cayley graph of $F$ restricted to balls in $F$ of size $N\le 10^7$, which corresponds to consider elements of $F$ of distance up to $14$ from the identity element. Moreover, they computed Cohen's cogrowth coefficients $\gamma_{2n}$ for $n=1,2,\ldots,11$ (cf.~Table~3 in \cite{ElderRechnitzerWongCogrowth}). Note that $\gamma_{2n}$ is equal to our ``reduced" number $\eta_n$ (cf.~Section~\ref{s:three}). In the notation of \cite{ElderRechnitzerWongCogrowth}, $\eta_n=p_{2n}$, and $m_n=r_{2n}$, and they use a different method (based on power series) to pass from the $(\eta_n)$-series to the $(m_n)$-series.

\section{\textbf{Estimating Spectral distribution measures}}\label{s:fourB}

In this section, we will estimate the spectral distribution measures $\mu_{\tilde h}$ for $h=I+A+B$ (Case 1) and $h=A+A^{-1}+B+B^{-1}$ (Case 2), based on the moment sequences listed in Section~\ref{s:three}.
This is done by computing the (possibly signed) measures $\nu_N$ given by
$d\nu_N(t)=\rho_N(t)dt$
where $\rho_N$ is the unique polynomial on $\R$ of degree $2N$ for which
\begin{equation}\label{e:4.munyapprox}
  \int_J t^n\, d \mu_{\tilde h}(t)=\int_J t^n\,d \nu_N(t), \qquad 0\le n\le 2N
\end{equation}
where $J=[-3,3]$ in case 1, and $J=[-4,4]$ in  case 2.
We do not know, whether $\mu_{\tilde h}$ has a density with respect to the Lebesgue measure, but if it has a density $f$ and if $f$ happens to be in $L^2(J,dt)$, then $\rho_N$ is simply the orthogonal projection of $f$ onto the subspace of $L^2(J,dt)$ spanned by $1$, $t$, $t^2$, $\ldots$, $t^{2N}$.
Let $(P_n(t))_{n=0}^\infty$ be the sequence of Legendre polynomials.
Then it is is well known that the sequence
$\left(\sqrt{n+\frac12} P_n(t)\right)_{n=0}^\infty$ form an orthonormal basis for $L^2([-1,1],dt)$. Hence
$$p_n(t):=\sqrt{\frac{n+\frac12}{q+1}}P_n(\frac t {q+1}),\qquad n\in\N_0$$
form an orthonormal basis for $L^2(J,dt)$, where $q=2$ in case 1 and $q=3$ in case 2.
Hence in the case $\mu_{\tilde h}$ has density $f\in L^2(J,dt)$, $\rho_N$ is the orthogonal projection of $f$ onto $\Span\{1,t,t^2,\ldots,t^{2N}\}$, i.e.
\begin{equation}\label{e:4.rhoN}
  \rho_N(t)=\sum_{n=0}^{2N} \langle f, p_n\rangle p_n(t)
\end{equation}
where
$$\langle f,p_n\rangle = \int_{-(q+1)}^{q+1} f(s) p_n(s) ds$$
Note that (\ref{e:4.rhoN}) can also be written as
\begin{equation}\label{e:4.rhoNseconddefn}
  \rho_N(t)=\sum_{n=0}^{2N}\left(\int_{-(q+1)}^{q+1} p_n(s)\,d\mu_{\tilde h}(s)\right)p_n(t).
\end{equation}
The latter formula also makes sense if $\mu_{\tilde h}$ does not have an $L^2$-density, and it is not hard to check that (\ref{e:4.rhoNseconddefn}) provides the unique solution to (\ref{e:4.munyapprox}) also when $\mu_{\tilde h}$ does not have an $L^2$-density.
Based on our moment calculations we can compute the polynomials $\rho_N(t)$ for $1\le N\le 37$ in case 1 and for $1\le N\le 24$ in case 2.
Since $\mu_{\tilde h}$  is a symmetric measure, all the odd terms in (\ref{e:4.rhoNseconddefn})  are zero. Hence $\rho_N$ is an even polynomial of degree at most $2N$.

Case 1: In Fig.~\ref{f:IABLebesgueDensity} we have for $h=I+A+B$ plotted $\rho_{37}$ for $0\le t\le 3$ together with the corresponding ``free" density (cf.~Section~\ref{s:one}) and in Figs.~\ref{f:IABLebesgueDensityTail}-\ref{f:IABLebesgueDensityTailSumOddEven} we have plotted first $\rho_{36}$ and $\rho_{37}$ in the tail interval $[2\sqrt 2, 3]$ and next $\frac 12 (\rho_{36}+\rho_{37})$ in the same interval.
The reason is that the plot in Fig.~\ref{f:IABLebesgueDensityTail} shows that for $t$ close to $3$, $\rho_{36}$ and $\rho_{37}$ oscillates around zero with opposite signs, and therefore we expect that $\frac12(\rho_{36}(t)+\rho_{37}(t))dt$ gives a better approximation to the measure $\mu_{\tilde h}$.
Figure~\ref{f:IABLebesgueDensityTailSumOddEven} indicates, that $\mu_{\tilde h}$ has very little mass on the interval $[2.9,3.0]$, and hence
$||h||=\max(\sup(\mu_{\tilde h}))$ could be any number in the interval $[2.9,3.0]$, including the number 2.95 found in Section~\ref{s:four} by an extrapolation argument.  Recall from the end of Section~\ref{s:two} that $\mu_{h^*h}$ is the image measure of $\mu_{\tilde h}$ by the map $t\mapsto t^2$.
Hence Figs.~\ref{f:IABLebesgueDensity}-\ref{f:IABLebesgueDensityTailSumOddEven} can also be used to compute an approximation to the spectral density of $(I+A+B)^*(I+A+B)$.

\begin{figure}[h]
\centerline{
\includegraphics[scale=.7]{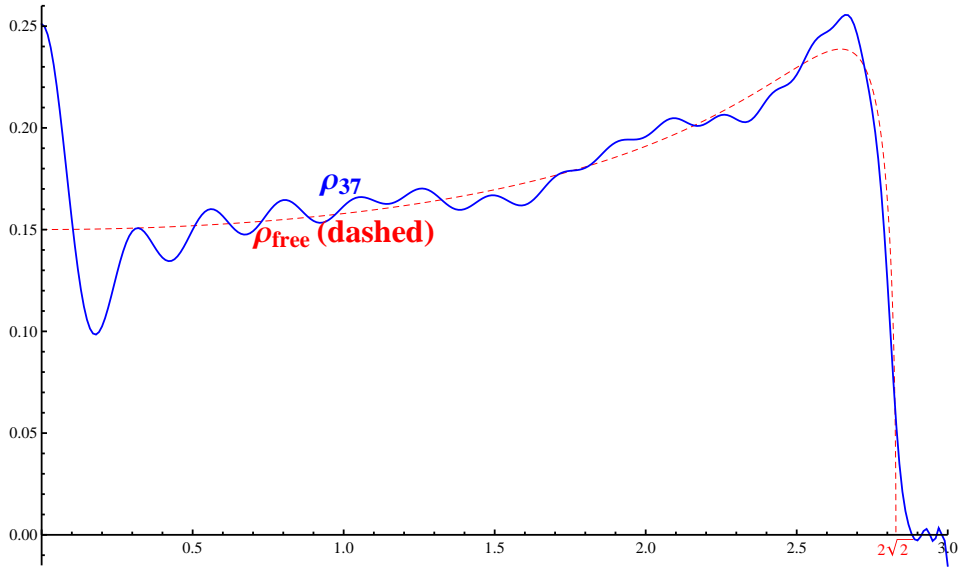}
}
\vspace*{8pt}
  \caption{Estimating the Lebesgue density for $\mu_{\tilde h}$, where $h=I+A+B$ \label{f:IABLebesgueDensity}}
\end{figure}
\begin{figure}[H]
\centerline{
\includegraphics[scale=.57]{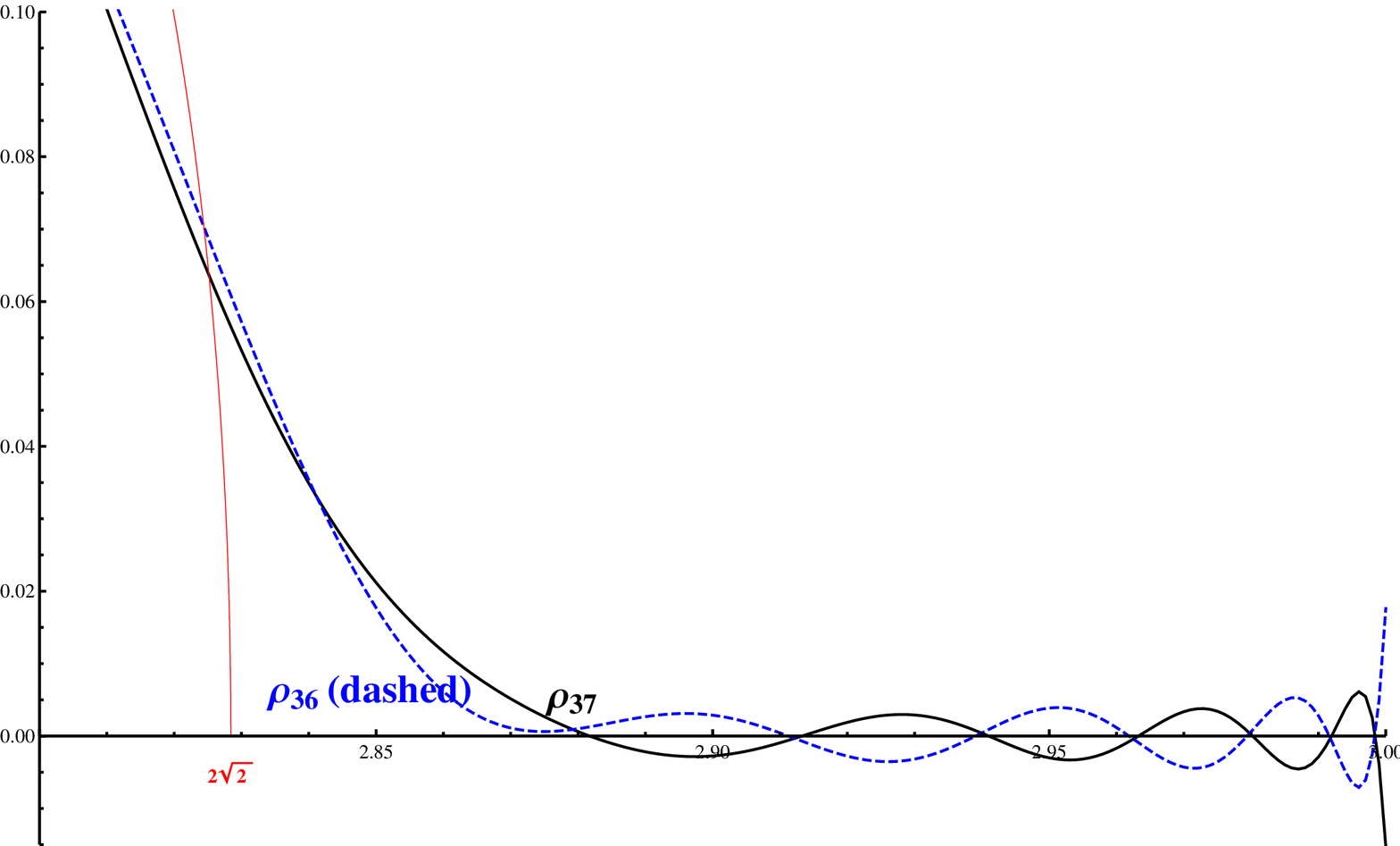}
}
  \caption{Estimating the tail of the Lebesgue density for $\mu_{\tilde h}$, where $h=I+A+B$  \label{f:IABLebesgueDensityTail}}
\end{figure}
\begin{figure}[H]
\centerline{
\includegraphics[scale=.57]{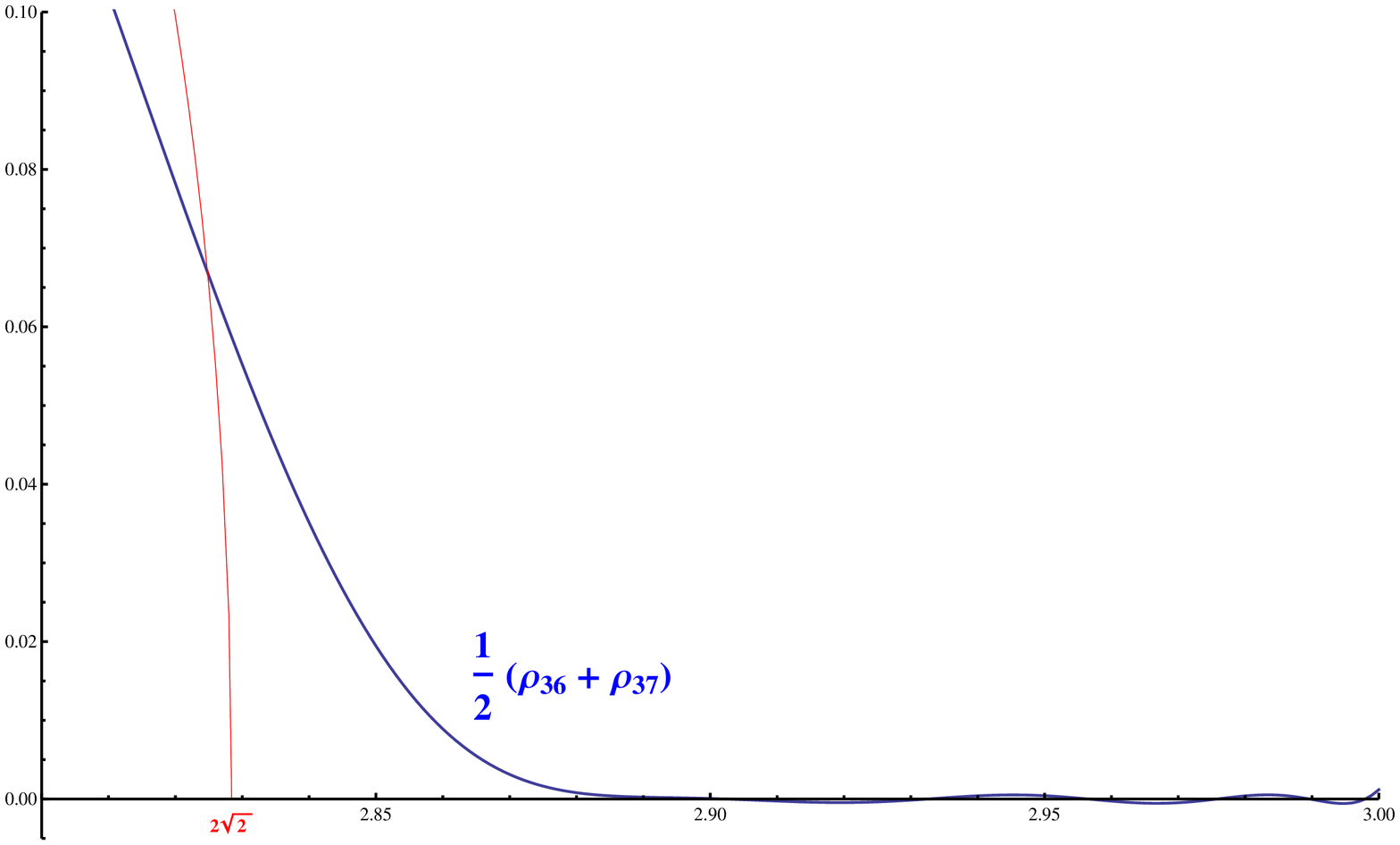}
}
\vspace*{8pt}
  \caption{A better estimate of the Lebesgue density for $\mu_{\tilde h}$, where $h=I+A+B$  \label{f:IABLebesgueDensityTailSumOddEven}}
\end{figure}

\begin{figure}[h]
\centerline{
\includegraphics[scale=.69]{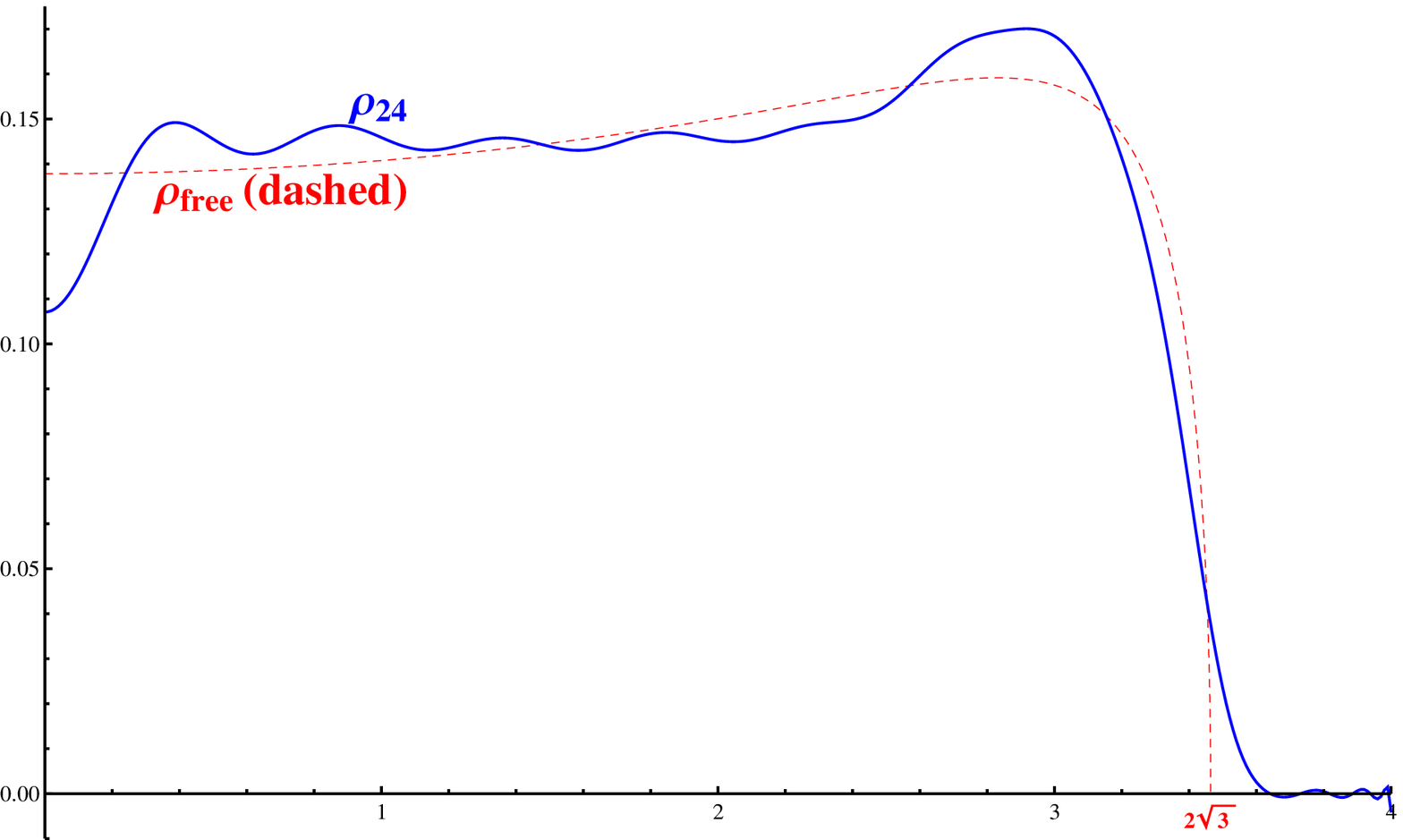}
}
\vspace*{8pt}
  \caption{Estimating the Lebesgue density for $\mu_{\tilde h}=\mu_h$, where $h=A+B+A^{-1}+B^{-1}$ \label{f:ABAinvBinvLebesgueDensity}}
\end{figure}
\begin{figure}[h]
\centerline{
\includegraphics[scale=.69]{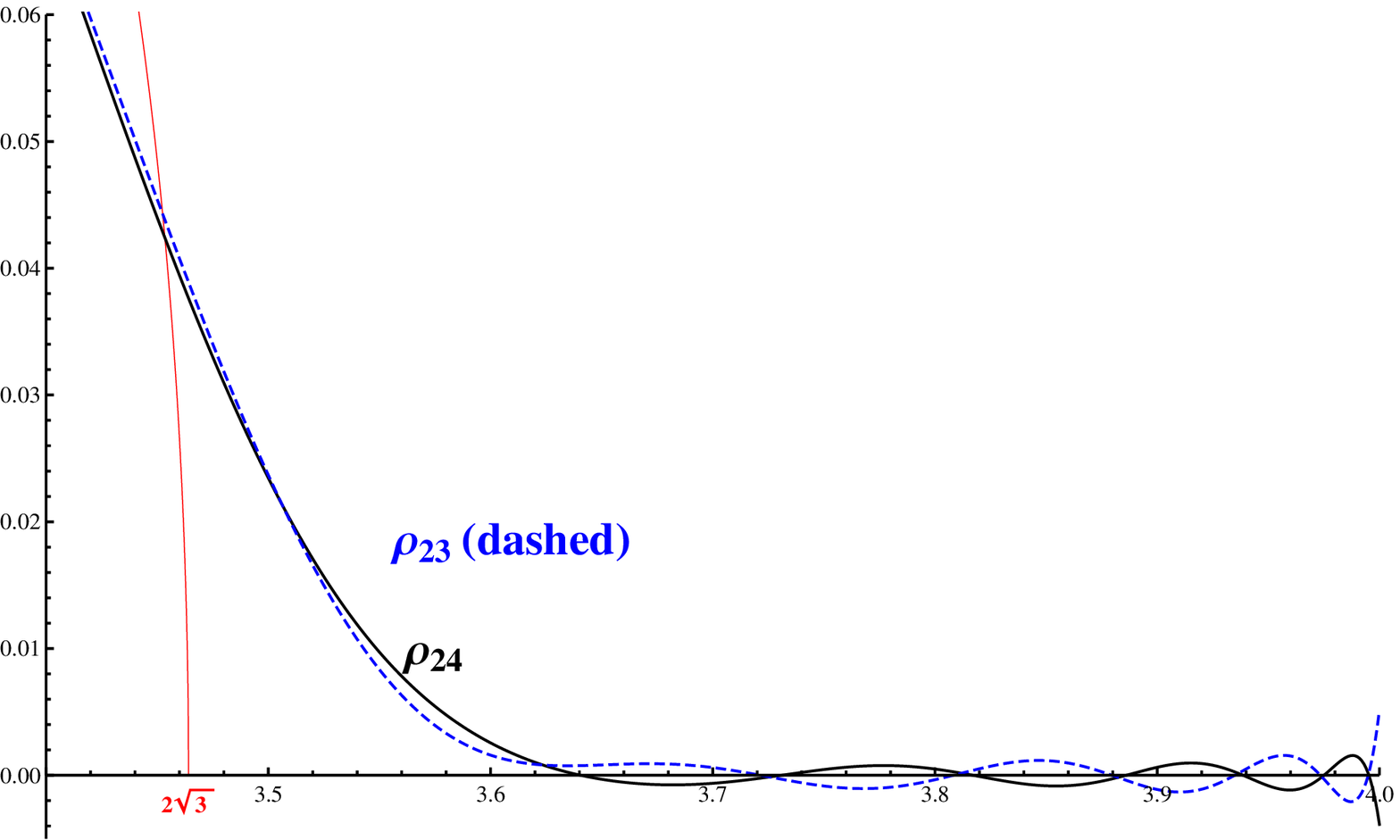}}
  \caption{Estimating the tail of the Lebesgue density for $\mu_{\tilde h}=\mu_h$, where $h=A+B+A^{-1}+B^{-1}$ \label{f:ABAinvBinvLebesgueDensityTail}}
\end{figure}
\begin{figure}[h]
\centerline{
\includegraphics[scale=.65]{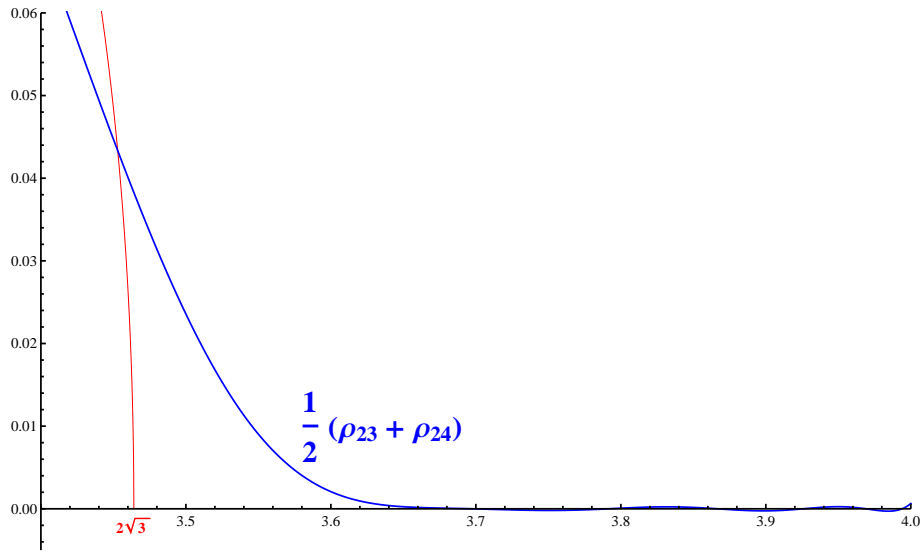}}
  \caption{A better estimate of the Lebesgue density for $\mu_{\tilde h}=\mu_h$, where $h=A+B+A^{-1}+B^{-1}$ \label{f:ABAinvBinvLebesgueDensityTailSumOddEven}}
\end{figure}
\newpage
Case 2:
In Fig.~\ref{f:ABAinvBinvLebesgueDensity} we have for $h=A+A^{-1}+B+B^{-1}$ plotted $\rho_{24}$ for $0\le t\le 4$ together with the corresponding ``free" density, and in Figs.~\ref{f:ABAinvBinvLebesgueDensityTail}-\ref{f:ABAinvBinvLebesgueDensityTailSumOddEven} we have plotted $\rho_{23}$, $\rho_{24}$ and $\frac 12(\rho_{23}+\rho_{24})$ in the tail interval $[2\sqrt 3, 4]$. Again
$\frac 12(\rho_{23}(t)+\rho_{24}(t))dt$ appear to give the best approximation to the measure $\mu_{\tilde h}$ in the tail region, and it shows that $\mu_{\tilde h}$ must have very little mass in the interval $[3.7,4.0]$.
Hence $||h||=\max(\sup(\mu_{\tilde h}))$ could be any number in the interval $[3.7,4.0]$ including the number $3.87$ found in Section~\ref{s:four} by an extrapolation argument.
Note finally, that in the case of $h=A+A^{-1}+B+B^{-1}$, $\mu_h=\mu_{\tilde h}$. To see this, we just have to check that $\mu_h$ and $\mu_{\tilde h}$ have the same moments. The even moments coincide because $h=h^*$ and therefore
$$m_{2n}(\tilde h)=\tau((h^*h)^n)=\tau(h^{2n})=m_{2n}(h),\qquad n\in \N_0.$$
Moreover,
$$m_{2n+1}(\tilde h)=0=m_{2n+1}(h), \qquad n\in\N_0,$$
where the first equality follows because $\mu_{\tilde h}$ is symmetric and the second equality follows because no word of odd length in $A,A^{-1},B,B^{-1}$ can represent the identity in $F$.

\section{\textbf{Relations between $||h_n||_2^2$,  $\xi_n$, $\eta_n$, $\zeta_n$ and $m_n$}}\label{s:five}

\subsection{\underline{The polynomials $Q_n$ with constant $q\in\N$}}\label{sss:QnConstantq}
Let $q\in\N$ be a fixed natural number.
Define the polynomials $(Q_n)_{n\in\N}$ in $\C[t]$ recursively by
\begin{eqnarray}\label{e:QnConstantq}
  &&Q_1(t)=t\\
  \nonumber &&Q_2(t)=t^2-(q+1)\\
  \nonumber &&Q_{n+1}(t)=t Q_n(t)-q Q_{n-1}(t), \qquad n\ge 2.
\end{eqnarray}
 Since $Q_n$ is an even polynomial for $n$ even and an odd polynomial for $n$ odd, there exists polynomials $(Q_m^{(1)})_{m\ge1}$ and $(Q_m^{(2)})_{m\ge0}$ in $\C[t]$ such that
\begin{eqnarray}\label{e:Q1nQ2n}
  &&Q_{2m}(t)=Q_m^{(1)}(t^2) \qquad m\ge 1\\
\nonumber  &&Q_{2m+1}(t)=tQ_m^{(2)}(t^2)\qquad m\ge 0.
\end{eqnarray}
Next we write the $Q_n$ polynomials in terms of Chebyshev polynomials.

\begin{pro}\label{p:QnRelatedToChebyshevPolys}

  Let $(T_n)_{n\in\N}$ and $(U_n)_{n\in\N}$ be the Chebyshev polynomials of first and second kind. Then
\begin{itemize}
  \item [$(i)$] $  Q_n(t)=\left(\frac 2q \,T_n\left(\frac t{2\sqrt{q}}\right)+\frac{q-1}q\,\, U_n\left(\frac t{2\sqrt{q}}\right)\right)q^{n/2},\qquad n\in\N,\,\,t\in\R.$
  \item [$(ii)$] $ Q_{2n}(t)-(q-1)\sum\limits_{k=1}^{n-1}Q_{2k}(t)=(q-1)+2T_{2n}\left(\frac t{2 \sqrt q}\right) q^n,\qquad n\in\N,\,\,t\in\R$.
\end{itemize}
\end{pro}

\begin{proof}
Let $Q_0(t):=\frac {q+1}q$. Then the recursion formula
\begin{eqnarray*}
Q_{n+1}(t)=t Q_n(t)-q Q_{n-1}(t)
\end{eqnarray*}
holds for all $n\ge 1$.
Letting $$q_n(t):=Q_n(2\sqrt{ q} t)q^{-n/2},\qquad n\ge 0,\,\, t\in\R,$$
we get
\begin{eqnarray*}
&&  q_0(t)=\frac {q+1}q\\
&& q_1(t)=2t\\
&&q_{n+1}(t)=2t q_n(t)-q_{n-1}(t), \qquad n\ge 1.
\end{eqnarray*}
The Chebyshev polynomials also satisfy an identical recursion formula:
\begin{eqnarray*}
&&T_{n+1}(t)=2t T_n(t)-T_{n-1}(t), \qquad n\ge 1,\\
&&U_{n+1}(t)=2t U_n(t)-U_{n-1}(t), \qquad n\ge 1.
\end{eqnarray*}
Hence, if we can choose constants $\alpha,\beta\in\R$ such that
\begin{eqnarray*}
q_0(t)=\alpha T_0(t)+\beta U_0(t)\qquad q_1(t)=\alpha T_1(t)+\beta U_1(t),
\end{eqnarray*}
then $q_n(t)=\alpha T_n(t)+\beta U_n(t)$ would hold for all $n\ge 0$.
Indeed,  $\alpha=2/q$ and $\beta=(q-1)/q$ is the only solution since $T_0(t)=1$, $T_1(t)=t$, $U_0(t)=1$, $U_1(t)=2t$.
This proves ($i)$ because
\begin{equation}\label{e:Qnrelatedtoqn}
  Q_n(t)=q_n(\frac t{2\sqrt{ q}})q^{n/2},\qquad n\ge 0,\, t\in\R.
\end{equation}
The proof of  $(ii)$ is an adaptation from the proof of formula (2.6) in \cite{HaagerupKnydby2013}.
From Eq.~(2) p. 184 in \cite{HTF2}, the Chebyshev polynomials satisfy
$$T_n(\cos\theta)=\cos n\theta, \qquad U_n(\cos\theta)=\frac{\sin(n+1)\theta}{\sin\theta}.$$
It follows that
$$U_n(t)-U_{n-2}(t)=2T_n(t),\quad n\ge2,$$
for $-1<t<1$ and hence for all $t\in\R$ because we are dealing with polynomials. Since $q_n(t)=\frac 2 q T_n(t)+\frac{q-1}q U_n(t)$, $t\in\R$, $n\ge0$, we get
$$q_n(t)-q_{n-2}(t)=2T_n(t)-\frac 2 q T_{n-2}(t),\qquad t\in\R,\,\, n\ge 2.$$
From this we get
$$\sum_{k=1}^n (q_{2k}(t)-q_{2k-2}(t))q^{k-n}=2T_{2n}(t)-\frac2{q^n}T_0(t)=2T_{2n}(t)-\frac 2{q^n}.$$
Simplifying,
$$q_{2n}(t)-(q-1)\sum_{k=1}^{n-1} \frac{q_{2k}(t)}{q^{n-k}}-\frac{q_0(t)}{q^{n-1}}=2T_{2n}(t)-\frac 2{q^n}.$$
Therefore
$$  q_{2n}(t)-(q-1)\sum_{k=1}^{n-1} \frac{q_{2k}(t)}{q^{n-k}}=2T_{2n}(t)+\frac {q-1}{q^n}.$$
The last equation together with Eq.~(\ref{e:Qnrelatedtoqn}) yields $(ii)$.
\end{proof}

\subsection{\underline{The group ring sequences $(h_n)$, $(k_n)$}}
  Let $\Gamma$ be a discrete group, let $Y\subset \Gamma$ be a finite set with $|Y|=q+1$ elements $(q\in\N)$, and let
$h=\sum_{s\in Y} s$ as in Section~\ref{s:three}.
  Define $Q_n$, $Q_n^{(1)}$, $Q_n^{(2)}$ as in Eqs.~(\ref{e:QnConstantq}) and (\ref{e:Q1nQ2n}) for this value of $q$.
  Define the sequences $(h_n)_{n\in\N}$, $(k_n)_{n\in\N}$ in the group ring $\C\Gamma$ of $\Gamma$ by
  \begin{eqnarray}\label{e:hn}
   &&E_n:=\{(s_1,\ldots,s_n)\in Y^n\mid  s_1\ne s_2\ne\ldots\ne s_n\}. \\
   \nonumber&&h_n:=\sum_{(s_1,\ldots,s_n)\in E_n} s_1(s_2^{-1}s_3\cdots s_{n-1}^{-1}s_n) \qquad (n \text{ odd})\\
   \nonumber&&h_n:=\sum_{(s_1,\ldots,s_n)\in E_n} s_1^{-1}s_2\cdots s_{n-1}^{-1}s_n \qquad (n \text{ even}).
  \end{eqnarray}
    \begin{eqnarray}\label{e:kn}
   &&k_n:=\sum_{(s_1,\ldots,s_n)\in E_n} s_1^{-1}(s_2s_3^{-1}\cdots s_{n-1}s_n^{-1}) \qquad (n \text{ odd})\\
   \nonumber&&k_n:=\sum_{(s_1,\ldots,s_n)\in E_n} s_1s_2^{-1}\cdots s_{n-1}s_n^{-1} \qquad (n \text{ even}).
  \end{eqnarray}
In all cases, $s_1\ne s_2\ne\cdots\ne s_n$ means $s_i\ne s_{i+1}$ $(1\le i\le n-1)$, and $Y^{-1}:=\{x^{-1}:x\in Y\}$.
Notice that $h_1=h$ and $k_1=h^*$. In fact $h_n^*=k_n$ for $n$ odd, and for $n$ even $h_n^*=h_n$ and $k_n^*=k_n$.

\begin{thm}\label{t:hnConnectedToQn}

Let $(h_n)_{n\in\N}$, $(k_n)_{n\in\N}$ be the sequences defined in Eqs.~(\ref{e:hn}) and (\ref{e:kn}). Then in the algebra $M_2(\C\Gamma)$ of $2\times2$ matrices  over $\C\Gamma$ we have
\begin{itemize}
  \item [($i$)]
  $\left(
                   \begin{array}{cc}
                     0 & k_n \\
                     h_n & 0 \\
                   \end{array}
                 \right)
                 =
                 Q_n\left(
                      \begin{array}{cc}
                        0 & h^* \\
                        h & 0 \\
                      \end{array}
                    \right)
  $ for  $n$ odd.
  \item [($ii$)]
  $\left(
                   \begin{array}{cc}
                     h_n & 0 \\
                     0 & k_n \\
                   \end{array}
                 \right)
                 =
                 Q_n\left(
                      \begin{array}{cc}
                        0 & h^* \\
                        h & 0 \\
                      \end{array}
                    \right)
  $ for  $n$ even.
\end{itemize}
These formulas can also be expressed in the group ring $\C(\Gamma)$ as
\begin{itemize}
  \item [($iii$)]  $h_{2m+1}=h Q_m^{(2)}(h^*h)$,\\
   $k_{2m+1}=h^*Q^{(2)}_m(hh^*)$ \qquad $m\ge0$.
  \item [($iv$)]   $h_{2m}=Q_m^{(1)}(h^*h)$,\\
  $k_{2m}=Q_m^{(1)}(hh^*)$ \qquad $m\ge 1$.
\end{itemize}
\end{thm}
\begin{proof}
  We start by proving $(i)$ and $(ii)$ by induction in $n\in\N$.
  Note first that $(i)$ holds for $n=1$  and that $(ii)$ holds for $n=2$ because
  $$ h_2=(\sum_{s_1\in Y}s_1^{-1})(\sum_{s_2\in Y} s_2)-|Y|e=h^*h-(q+1)e,$$
  and similarly $k_2=hh^*-(q+1)e$.
  We next prove the following 4 recursion formulas.
  \begin{equation}\label{e:hnplus1knplus1Even}
    \begin{array}{ll}
      h_{n+1}=h h_n-q h_{n-1}, & \hbox{$n\ge 2$, $ n$  even} \\
      k_{n+1}=h^*k_n-qk_{n-1}, & \hbox{$n\ge 2$, $n$  even.}
    \end{array}
  \end{equation}
    \begin{equation}\label{e:hnplus1knplus1Odd}
    \begin{array}{ll}
      h_{n+1}=h^* h_n-q h_{n-1}, & \hbox{$n\ge 3$, $ n$  odd} \\
      k_{n+1}=h k_n-qk_{n-1}, & \hbox{$n\ge 3$, $n$  odd.}
    \end{array}
  \end{equation}
  To prove the first formula in Eq.~(\ref{e:hnplus1knplus1Even}) note that $n+1$ is odd. Hence
  \begin{eqnarray*}
 h_{n+1}&=&\sum_{(s_0,\ldots,s_n)\in E_{n+1}} s_0(s_1^{-1}s_2\ldots s_{n-1}^{-1}s_n)\\
    &=&\sum_{(s_1,\ldots,s_n)\in E_{n}}\sum_{s_0\in Y\backslash\{s_1\}} s_0(s_1^{-1}s_2\ldots s_{n-1}^{-1}s_n)\\
    &=&\sum_{(s_1,\ldots,s_n)\in E_{n}}\big(\sum_{s_0\in Y} s_0(s_1^{-1}s_2\ldots s_{n-1}^{-1}s_n)- s_1(s_1^{-1}s_2\ldots s_{n-1}^{-1}s_n)\big)\\
    &=&\big(\sum_{s_0\in Y} s_0 \big)\big(\sum_{(s_1,\ldots,s_n)\in E_n} s_1^{-1}s_2\ldots s_{n-1}^{-1}s_n\big)-\\
    &&\qquad \sum_{(s_2,\ldots,s_n)\in E_{n-1}}\big( \sum_{s_1\in Y\backslash\{s_2\}} s_1s_1^{-1}\big)(s_2\cdots s_{n-1}^{-1}s_n)\\
    &=&hh_n-qh_{n-1}.
  \end{eqnarray*}
The second formula in Eq.~(\ref{e:hnplus1knplus1Even}) follows from this by replacing $Y$ with $Y^{-1}$.
The two formulas in Eq.~(\ref{e:hnplus1knplus1Odd}) can be proven in exactly the same way noticing that in this case $n+1$ is even.
The formulas in Eqs.~(\ref{e:hnplus1knplus1Even}) and (\ref{e:hnplus1knplus1Odd}) can be rewritten as
\begin{eqnarray*}
\left(
  \begin{array}{cc}
    0 & k_{n+1} \\
    h_{n+1} & 0 \\
  \end{array}
\right)=
\left(
  \begin{array}{cc}
    0 & h^* \\
    h & 0 \\
  \end{array}
\right)
\left(
  \begin{array}{cc}
    h_n & 0 \\
    0 & k_n \\
  \end{array}
\right)
-q
\left(
  \begin{array}{cc}
    0 & k_{n-1} \\
    h_{n-1} & 0 \\
  \end{array}
\right)
\end{eqnarray*}
for $n\ge 2$, $n$ even; and
\begin{eqnarray*}
\left(
  \begin{array}{cc}
    h_{n+1} & 0 \\
    0 & k_{n+1} \\
  \end{array}
\right)=
\left(
  \begin{array}{cc}
    0 & h^* \\
    h & 0 \\
  \end{array}
\right)
\left(
  \begin{array}{cc}
    0 & k_n \\
    h_n & 0 \\
  \end{array}
\right)
-q
\left(
  \begin{array}{cc}
    h_{n-1} & 0 \\
    0 & k_{n-1} \\
  \end{array}
\right)
\end{eqnarray*}
for $n\ge 3$, $n$ odd.
By induction in $n$ we get
$$Q_{n+1}(t)=tQ_n(t)-qQ_{n-1}(t),\qquad n\ge 2,$$
once we rewrite the above two matrix equations in terms of $(i)$ and $(ii)$,
which is precisely the definition of $Q_n$ given in Formula~$(\ref{e:QnConstantq})$.
Hence $(i)$ and $(ii)$ hold.
Since
$$\left(
    \begin{array}{cc}
      0 & h^* \\
      h & 0 \\
    \end{array}
  \right)^2=
  \left(
    \begin{array}{cc}
      h^*h & 0 \\
      0 & hh^* \\
    \end{array}
  \right),$$
we have by Eq.~(\ref{e:Q1nQ2n})
\begin{eqnarray*}
\left(
    \begin{array}{cc}
      0 & k_{2m+1} \\
      h_{2m+1} & 0 \\
    \end{array}
  \right)
  &=&
  Q_{2m+1}\left(
    \begin{array}{cc}
      0 & h^* \\
      h & 0 \\
    \end{array}
  \right)\\
  &=&
  \left(
    \begin{array}{cc}
      0 & h^* \\
      h & 0 \\
    \end{array}
  \right) Q_m^{(2)}\left(
    \begin{array}{cc}
      h^*h & 0 \\
      0 & hh^* \\
    \end{array}
  \right)\\
  &=&
    \left(
    \begin{array}{cc}
      0 &  h^*Q_m^{(2)}(hh^*) \\
       h Q_m^{(2)}(h^*h) & 0 \\
    \end{array}
  \right)
\end{eqnarray*}
   for $m\ge 0$, proving $(iii)$. Similarly for $m\ge 1$, we have
\begin{eqnarray*}
  \left(
    \begin{array}{cc}
      h_{2m} & 0 \\
      0 & k_{2m} \\
    \end{array}
  \right)
  =
  Q_{2m}\left(
    \begin{array}{cc}
      0 & h^* \\
      h & 0 \\
    \end{array}
  \right)
  =
  \left(
    \begin{array}{cc}
      Q_m^{(1)}(h^*h) & 0 \\
      0 & Q_m^{(1)}(hh^*) \\
    \end{array}
  \right)
\end{eqnarray*}
proving $(iv)$.
\end{proof}

\begin{rem}

  If $Y=Y^{-1}$ then $h=h^*$ and $h_n=k_n=Q_n(h)$ for all $n\in\N$.
\end{rem}

\begin{pro}\label{p:hnstarhneqh2nplus}

  For all $n\in\N$ we have
  \begin{equation}
    h_n^*h_n=h_{2n}+(q+1)q^{n-1}e+(q-1)\sum_{i=1}^{n-1}q^{i-1} h_{2n-2i}.
  \end{equation}
\end{pro}

\begin{proof}
  For $n=1$ we have from the proof of Theorem~\ref{t:hnConnectedToQn} that
  $h^*h=h_2+(q+1)e$.
  Consider now $n\ge 2$. If $n$ is odd, then
  \begin{eqnarray*}
    h_n^*h_n&=&\sum_{\substack{(t_1,\ldots,t_n)\in E_n\\ (s_1,\ldots,s_n)\in E_n}} (t_n^{-1}t_{n-1}\cdots t_3^{-1}t_2)t_1^{-1}s_1(s_2^{-1}s_3\cdots s_{n-1}^{-1}s_n)=\sum_{k=0}^n a_k,
  \end{eqnarray*}
  where $a_k$ is obtained by summing over only those $(s_1,\ldots,s_n),(t_1,\ldots,t_n)\in E_n$ for which
$(s_1,\ldots,s_k)=(t_1,\ldots,t_k)$ and $s_{k+1}\ne t_{k+1}$
when $0\le k\le n-1$, and $(s_1,\ldots,s_n)=(t_1,\ldots,t_n)$ when $k=n$.
If $k=0$ then $s_1\ne t_1$ and thus $a_0=h_{2n}$.
If $k=n$ then $(t_1,\ldots,t_{n})=(s_1,\ldots,s_{n})$ and thus $a_n=|E_n|e=(q+1)q^{n-1}e$.
For  $1\le k\le n-1$, $k$ odd, we have
  \begin{eqnarray*}
a_k&=& \sum_{\substack{(t_{k+1},\ldots,t_n)\in E_{n-k}\\(s_{k+1},\ldots,s_n)\in E_{n-k}\\ t_{k+1}\ne s_{k+1}}} \sum_{\substack{(t_1,\ldots,t_k)\in E_k\\(s_1,\ldots,s_k)\in E_k\\ (t_1,\ldots,t_k)=(s_1,\ldots,s_k)\\ t_{k}\ne t_{k+1}\\ s_{k}\ne s_{k+1}}}(t_n^{-1}t_{n-1}\cdots t_{k+2}^{-1}t_{k+1})(s_{k+1}^{-1}s_{k+2}\cdots s_{n-1}^{-1}s_n)\\
&=& (q-1) q^{k-1}h_{2n-2k}
  \end{eqnarray*}
because for fixed $(t_{k+1},\ldots,t_n)$, $(s_{k+1},\ldots,s_n)\in E_{n-k}$ with $t_{k+1}\ne s_{k+1}$, $(s_1,\ldots,s_k)=(t_1,\ldots, t_k)$ can be chosen in $(q-1)q^{k-1}$ ways, namely $s_k=t_k\in Y\backslash\{s_{k+1},t_{k+1}\}$ can first be chosen in $|Y|-2=q-1$ ways, and next $s_{k-1}=t_{k-1}$, $s_{k-2}=t_{k-2}$ etc. can each be chosen in $q=|Y|-1$ ways.
The same holds for $k$ even and/or $n$ even by obvious modifications of the above proof.
\end{proof}

\begin{pro}

  We have
  \begin{equation}\label{e:norm2hnEQnorm2kn}
    ||h_n||_2=||k_n||_2,\qquad n\in\N
  \end{equation}
\end{pro}
\begin{proof}
  For $n$ odd,  $h_n^*=k_n$ by Theorem~\ref{t:hnConnectedToQn}. Thus $||h_n||_2=\tau(h_n^*h_n)^{1/2}=\tau(h_nh_n^*)^{1/2}=||k_n||_2$.
  For $n$ even, $n=2m$,  
Theorem~\ref{t:hnConnectedToQn} yields
  $||h_n||_2^2=\tau((Q_m^{(1)}(h^*h))^2)=\tau((Q_m^{(1)}(hh^*))^2)=||k_n||_2^2$ as $\tau((h^*h)^j)=\tau((h^*h)^{j-1}(h^*h))=\tau(h(h^*h)^{j-1}h^*)=\tau((hh^*)^j)$ for any $j\in\N$.
\end{proof}

\subsection{\underline{The integer sequences $\xi_n$, $\eta_n$, $\zeta_n$}}
Define the group ring sequence $(z_n)_{n\in\N}$ by
\begin{equation}\label{e:zn}
z_n:=\sum_{\substack{(s_1,\ldots,s_{2n})\in \tilde E_{2n}}}  s_1^{-1}s_2\cdots s_{2n-1}^{-1} s_{2n},
\end{equation}
where $\tilde E_{k}:=\{(s_1,\ldots,s_k)\in E_k\mid s_1\ne s_k\}$. We could say that $z_n$ is the cyclic version of $h_{2n}$ defined in Eq.~(\ref{e:hn}). Recall that $E_n\subset Y^n$, where $Y\subset \Gamma$ is a finite set  with $|Y|=q+1$ elements. We assume $q\ge 1$.
Define the number sequences $(\xi_n)_{n\in\N}$, $(\eta_n)_{n\in\N}$, $(\zeta_n)_{n\in\N}$ by
\begin{eqnarray}\label{e:xirhozetaN}
  &&\xi_n:=||h_n||_2^2-(q+1)q^{n-1}\\
\nonumber  &&\eta_n:=\tau(h_{2n})\\
\nonumber  &&\zeta_n:=\tau(z_n).
\end{eqnarray}

\begin{pro}\label{p:xietazetabounded}

For $n\in\N$, $\xi_n$, $\eta_n$, $\zeta_n$ are integers and
\begin{equation}
  0\le \zeta_n\le \eta_n\le \xi_n\le 4 q^{2n}.
\end{equation}
\end{pro}
\begin{proof}
  The numbers $\xi_n$, $\eta_n$, $\zeta_n$ are integers because $h_n\in\Z\Gamma$ and $\tau(x)\in\{0,1\}$ for all $x\in \Gamma$.
  Since $\tau(x)\ge0$ for all $x\in\Gamma$, and $h_n$ is a sum on $E_{2n}$ while $z_n$ is the same sum but on a subset of $E_{2n}$, we have $\eta_n\ge \zeta_n\ge 0$. By Proposition~\ref{p:hnstarhneqh2nplus}
$$ \xi_n=\tau(h_n^*h_n)-(q+1)q^{n-1}=\eta_n+(q-1)\sum_{i=1}^{n-1}q^{i-1} \eta_{n-i}\ge \eta_n.$$
Finally,
$$\xi_n\le ||h_n||^2=\tau(h_n^*h_n)\le \sum_{s,t\in E_n} 1=|E_n|^2=((q+1)q^{n-1})^2=\left(\frac{q+1}{q}\right)^2q^{2n}\le 4 q^{2n}.$$
\end{proof}

\begin{pro}\label{p:6.7page29A}

  $$\tau(h_{2n})=\tau(z_n)+(q-1)\sum_{k=1}^{n-1}q^{k-1}\tau(z_{n-k})$$
\end{pro}
\begin{proof}
  We have
$$h_{2n}=\sum_{(s_1,\ldots,s_{2n})\in E_{2n}} s_{2n}^{-1}s_{2n-1}\cdots s_2^{-1}s_1=\sum_{k=0}^n b_k$$
where
$$b_k=\sum_{(s_1,\ldots,s_{2n})\in E_{2n}^{(k)}} s_{2n}^{-1}s_{2n-1}\cdots s_2^{-1}s_1$$
and where for $1\le k\le n-1$, $E_{2n}^{(k)}$ denotes the subset of $(s_1,\ldots,s_{2n})\in E_{2n}$ for which
$$s_1=s_{2n}, s_2=s_{2n-1},\ldots,s_k=s_{2n-k+1},s_{k+1}\ne s_{2n-k}$$
and for $k\in \{0,n\}$:
$$E_{2n}^{(0)}=\{(s_1,\ldots,s_{2n})\in E_{2n}\mid s_1\ne s_{2n}\}=\tilde E_{2n},$$
$$E_{2n}^{(n)}=\{(s_1,\ldots,s_{2n})\in E_{2n}\mid s_i=s_{2n+1-i},\,\, i=1,\ldots,n\}.$$
Clearly $b_0=z_n$. Moreover $b_n=0$ because $E_{2n}^{(n)}=\emptyset$ as $s_n\ne s_{n+1}$. For $1\le k\le n-1$ ($k$ odd) we can write
$$b_k=\sum u_1^{-1}u_2\cdots u_k^{-1}(s_{k+1}s_{k+2}^{-1}\cdots s_{2n-k})u_k\cdots u_2^{-1}u_1$$
where the summation is over all
$(s_{k+1},s_{k+2},\ldots, s_{2n-k})\in \tilde E_{2n-2k}$ and $(u_1,u_2,\ldots,u_k)\in E_k$ for which $u_k\not\in\{s_{k+1},s_{2n-k}\}.$

For fixed $(s_{k+1},s_{k+2},\ldots, s_{2n-k})$ there are exactly $(q-1)q^{k-1}$ choices of $(u_1,\ldots, u_k)$, namely first $u_k$ can be chosen in $|Y|-2=q-1$ ways because $s_{k+1}\ne s_{2n-k}$ and next each of $u_{k-1}$, $u_{k-2}$, $\ldots$, $u_1$ can be chosen in $|Y|-1=q$ ways.
Since
$$\tau(u_1^{-1}u_2\ldots u_k^{-1}(s_{k+1}s_{k+2}^{-1}\cdots s_{2n-k})u_k\ldots u_2^{-1}u_1)=\tau(s_{k+1}s_{k+2}^{-1}\ldots s_{2n-k})$$
it follows that
\begin{eqnarray*}
  \tau(b_k)&=&(q-1)q^{k-1}\sum_{(s_{k+1},\ldots,s_{2n-k})\in \tilde E_{2n-2k}} \tau(s_{k+1} s_{k+2}^{-1}\cdots s_{2n-k})\\
  &=&(q-1)q^{k-1}\tau(z_{n-k}).
\end{eqnarray*}
The same formula holds for $k$ even ($1\le k\le n-1$) by an obvious modification of the proof.
This proves the proposition.
\end{proof}

\begin{pro} \label{p:xietazetaRelatedToEachOther}

For $n\ge1$ we have
\begin{itemize}
  \item [($i$)]  $\xi_n=\eta_n+(q-1)\sum \limits_{i=1}^{n-1} q^{i-1}\eta_{n-i}$
  \item [($ii$)] $\eta_n=\xi_n-(q-1)\sum\limits_{i=1}^{n-1} \xi_i$.
\end{itemize}
Similarly,
\begin{itemize}
  \item [($iii$)] $\eta_n=\zeta_n+(q-1)\sum\limits_{i=1}^{n-1}q^{i-1}\zeta_{n-i}$
  \item [($iv$)] $\zeta_n=\eta_n-(q-1)\sum\limits_{i=1}^{n-1}\eta_i.$
\end{itemize}
\end{pro}
\begin{proof}
  Consider the power series,
$$A(t):=\sum_{n=1}^{\infty} \xi_n t^n,\qquad B(t):=\sum_{n=1}^{\infty}\eta_n t^n, \qquad C(t):=\sum_{n=1}^{\infty} \zeta_n t^n.$$
By Proposition~\ref{p:xietazetabounded}, they are all convergent for all $t\in\C$ with $|t|<\frac 1{q^2}$.
We have already seen that $(i)$ follows from Proposition~\ref{p:hnstarhneqh2nplus}.
By $(i)$ we have for $|t|<\frac 1 {q^2}$:
$$A(t)=B(t)\big(1+(q-1)\sum_{k=0}^\infty q^kt^{k+1}\big)=B(t) \frac{1-t}{1-qt}.$$
Hence for $|s|<\tfrac 1{q^2}$, we have
$$B(t)=A(t)\frac{1-qt}{1-t}=A(t)\big(1-(q-1)\sum_{k=1}^nt^k\big)$$
By comparing the coefficients of $t^n$ in the power series expansion of $B(t)$ and of $A(t)\frac{1-qt}{1-t}$ we get ($ii$).
Note next that $(iii)$ follows from Proposition~\ref{p:6.7page29A}.
Hence, as in the proof of ($i$)$\Rightarrow$($ii$) we get
$$B(t)=C(t)\frac{1-t}{1-qt}\qquad\text{and}\qquad C(t)=\frac{1-qt}{1-t}B(t)$$
which implies $(iv)$.
\end{proof}


\subsection{\underline{The symmetric measure $\mu_{\tilde h}$}}
Let $Y\subset \Gamma$ be a finite subset with $|Y|=q+1$ elements and let as before
 $$h:=\sum_{s\in Y} s\in \C\Gamma$$
and
\begin{equation}\label{e:tildeh}
  \tilde h:=\left(
            \begin{array}{cc}
              0 & h^* \\
              h & 0 \\
            \end{array}
          \right)=\left(
            \begin{array}{cc}
              0 & \sum\limits_{s\in S} s^{-1} \\
              \sum\limits_{s\in S} s & 0 \\
            \end{array}
          \right)\in M_2(\C\Gamma).
\end{equation}
Recall that the trace $\tilde \tau=\tau\otimes\tau_2$ on $M_2(\C\Gamma)$ was defined in Eq.~(\ref{e:tautilde}).
Let $\mu_{\tilde h}$ be the spectral measure of $\tilde h$ on the interval $[-(q+1),q+1]$ with respect to the trace $\tilde \tau$, i.e. the unique probability measure on $[-(q+1),q+1]$ satisfying
$$\tilde\tau({\tilde h}^n)=\int_{-(q+1)}^{q+1} t^n \,d\mu_{\tilde h},\qquad n\in\N_0.$$
(See Section~\ref{s:two}).
Since
$$
\tilde h^{2m}=\left(
                 \begin{array}{cc}
                   (h^*h)^m & 0 \\
                   0 & (hh^*)^m \\
                 \end{array}
               \right)\qquad m\ge0
$$
and
$$
\tilde h^{2m+1}=\left(
                 \begin{array}{cc}
                  0& h^*(hh^*)^m  \\
                  h(h^*h)^m &0\\
                 \end{array}
               \right)\qquad m\ge0
$$
we have
$$\int_{-(q+1)}^{q+1} t^{2m} \,d\mu_{\tilde h}(t)=\frac {\tau((h^*h)^m)+ \tau((hh^*)^m)}2=\tau((h^*h)^m) , \qquad m\ge 0,$$
and
$$\int_{-(q+1)}^{q+1} t^{2m+1} \,d\mu_{\tilde h}(t)=0 , \qquad m\ge 0.$$
The latter condition implies via Riesz representation theorem that $\mu_{\tilde h}$ is symmetric. i.e.
$\mu_{\tilde h}=\check{\mu}_{\tilde h}$, where $\check\mu_{\tilde h}$ is the image measure of $\mu_{\tilde h}$ with respect to the map $t\mapsto -t$.
(This could also be shown algebraically).
Let
\begin{eqnarray}\label{e:mn}
  m_n&:=&\tau((h^*h)^n)\\
\nonumber&=&\int_{\mathsmaller{-(q+1)}}^{\mathsmaller{q+1}} t^{2n}\,d\mu_{\tilde h}(t).
\end{eqnarray}
Notice that $m_0=1$ and $m_1=q+1$.

\begin{pro}\label{p:etazetaInTermsOfChebyshevPols}

  For $n\in\N$, let $\eta_n$, $\zeta_n$ be as in Eq.~(\ref{e:xirhozetaN}), and $T_n$, $U_n$ be the Chebyshev polynomials of first and second kind.  Then
\begin{itemize}
  \item [$(i)$] $\eta_n=q^n\int_{\mathsmaller{-(q+1)}}^{\mathsmaller{q+1}} \frac 2 q T_{2n}\left(\frac t {2\sqrt q}\right)+\frac{q-1}q U_{2n}\left(\frac t{2\sqrt q}\right)\,d\mu_{\tilde h}(t)$
  \item [$(ii)$] $\zeta_n=(q-1)+2 q^{n}\int_{\mathsmaller{-(q+1)}}^{\mathsmaller{q+1}} T_{2n}\left(\frac t{2\sqrt q}\right)\,d\mu_{\tilde h}(t)$.
\end{itemize}
\end{pro}
\begin{proof}
  Since $\tau((h^*h)^m)=\tilde\tau(\tilde h^{2m})=\int_{\mathsmaller{-(q+1)}}^{\mathsmaller{q+1}} t^{2m}\,d\mu_{\tilde h}(t)$, $m\ge0$, we have for every polynomial $p\in\C[X]$:
$$\tau(p(h^*h))=\int_{\mathsmaller{-(q+1)}}^{\mathsmaller{q+1}} p(t^2)\,d\mu_{\tilde h}(t).$$
Hence, by Theorem~\ref{t:hnConnectedToQn}($iv$) and Eq.~(\ref{e:Q1nQ2n})
\begin{equation}\label{e:rhoeqintQ2ndmuhtilde}
  \eta_n=\tau(h_{2n})=\int_{\mathsmaller{-(q+1)}}^{\mathsmaller{q+1}} Q_n^{(1)}(t^2)\,d\mu_{\tilde h}(t)=\int_{\mathsmaller{-(q+1)}}^{\mathsmaller{q+1}} Q_{2n}(t)\,d\mu_{\tilde h}(t),\qquad n\ge1.
\end{equation}
By Proposition~\ref{p:QnRelatedToChebyshevPolys}($i$) we get ($i$), and by  Proposition~\ref{p:QnRelatedToChebyshevPolys}($ii$) and Proposition~\ref{p:xietazetaRelatedToEachOther}$(iv)$ we get ($ii$).
\end{proof}

Define
\begin{equation}\label{e:mqn}
  m_n^{(q)}:=\binom{2n}n q^n-(q-1)\sum_{k=0}^{n-1} \binom{2n}k q^k.
\end{equation}
Later on, we will show that these are the even moments for the measure in (\ref{e:muTildeh1}).

\begin{pro}\label{p:momentsContainingFreePart}

  For $n\ge 1$, we have
\begin{equation}\label{e:mnfromzetas}
  m_n=m_n^{(q)}+\sum_{k=0}^{n-1}\binom{2n}k q^k \zeta_{n-k}.
\end{equation}
\end{pro}
\begin{proof}
 By Euler's Formula and the binomial theorem, we have
$$\cos^{2n}\theta=\frac1{2^{2n}}\left(\binom{2n}n+2\sum_{k=0}^{n-1}\binom{2n}k\cos2(n-k)\theta\right).$$
Substituting $t=\cos\theta$, we get
$$t^{2n}=\frac1{2^{2n}}\left(\binom{2n}n+2\sum_{k=0}^{n-1}\binom{2n}k T_{2n-2k}(t)\right),$$
for $-1\le t\le 1$, and hence also for all $t\in\R$ as both sides of the equality are polynomials.
Substituting $t$ with $\tfrac t{2\sqrt q}$ we get
$$t^{2n}=\binom{2n}n q^n+2\sum_{k=0}^{n-1}\binom{2n}k T_{2n-2k}(\frac t{2\sqrt q})q^n,$$
Integrating both sides with respect to $\mu_{\tilde h}$ we get by Eq.~(\ref{e:mn}) and Proposition~\ref{p:etazetaInTermsOfChebyshevPols}($ii$) that
$$m_n=\binom{2n}nq^n+\sum_{k=0}^{n-1}\binom{2n}k(\zeta_{n-k}-q+1)q^k,$$
which proves the proposition.
\end{proof}
A simple reformulation of (\ref{e:mnfromzetas}) yields a formula for computing the cyclic numbers $\zeta_n$ from the moments $m_n$:

\begin{cor}

We have $\zeta_1=m_1-m_1^{(q)}$, and
$$\zeta_n=m_n-m_n^{(q)}-\sum_{k=1}^{n-1} \binom{2n}k q^k \zeta_{n-k}\qquad (n>1).$$
\end{cor}

\section{\textbf{Amenability, Leinert sets and Cogrowth}}\label{s:six}

\subsection{\underline{Leinert sets}}

\begin{defn}[\cite{Lei}, Definition III.B in \cite{AO}]

  A subset $Y$ of a group $\Gamma$ is called a Leinert set if for all $n\in\N$,   all tuples $(s_1,\ldots, s_{2n})\in E_{2n}$ satisfy
$$s_1^{-1}s_2s_2^{-1}\cdots s_{2n-1}^{-1}s_{2n}\ne e.$$
\end{defn}
By Theorem IIIF and Theorem IID(b) in \cite{AO}  we have
\begin{itemize}
  \item [($i$)] If $Y\subset \Gamma$ and $Y\cap Y^{-1}=\emptyset$ then $Y\cup Y^{-1}$ is a Leinert set if and only if $Y$ generates freely a copy of the free group $\F_{|Y|}$ with $|Y|$ generators inside $\Gamma$.
  \item [($ii$)] If $Y\subset \Gamma$ and $e\not\in Y$ then $Y\cup\{e\}$ is a Leinert set if and only if $Y$ generates freely a copy of $\F_{|Y|}$ inside $\Gamma$.
\end{itemize}
By the following theorem, due to Kesten and Lehner, the norm of $h$ from Section~\ref{s:five} is bounded by two values, the lower bound is related to Leinert sets and the upper bound to amenability.

\begin{thm}[\cite{Ke2},\cite{Leh}]\label{t:LowerboundLeinerSetsUpperBoundAmenability}

  Let $Y$ be a finite set in a discrete group $\Gamma$ with $|Y|=q+1$ elements $(q\ge 2)$.
Let  $h:=\sum_{s\in Y} s$. Then
\begin{itemize}
  \item [($i$)] $2\sqrt q\le ||h|| \le q+1.$
  \item [($ii$)] $||h||=q+1$ if and only if the subgroup $\Gamma_0:=\langle Y^{-1}Y \rangle\subset \Gamma$  generated by $Y^{-1}Y$  is amenable.
  \item [($iii$)] $||h||=2 \sqrt q$ if and only if $Y$ is a Leinert set.
\end{itemize}
\end{thm}
\begin{proof}
($i$): The upper bound is trivial since each $s$ is a unitary operator in $L(\Gamma)$.
The proof of the lower bound is the following:
Write $Y=\{s_1,\ldots,s_{q+1}\}$. Apply now Proposition 2 and Proposition 5 of \cite{Leh} to $G=\F_{q+1}$, $H=\Gamma$, and $\rho:G\to H$ the unique group homomorphism for which
$$\rho(t_i)=s_i, \qquad (1\le i\le q+1),$$
where $t_1,\ldots,t_{q+1}$  are the generators of $\F_{q+1}$. Then
$$||\sum_{s\in Y} s||_{L(\Gamma)} \ge ||\sum_{i=1}^{q+1} t_i||_{L(\F_{q+1})}=2\sqrt q.$$
\\($ii$):
  We have $$h^*h=\sum_{s_1,s_2\in Y} s_1^{-1}s_2=\sum_{s\in Y^{-1}Y} c_s s,$$
where $c_{s^{-1}}=\bar{c}_s=c_s$ for all $s\in Y^{-1}Y$ because $h^*h$ is self-adjoint.
Moreover, $Y^{-1}Y$ is a symmetric set (i.e. $Y^{-1}Y=(Y^{-1}Y)^{-1}$), and
$$\sum_{s\in Y^{-1}Y} c_s=|Y|^2=(q+1)^2.$$
Hence, by Section 3 of \cite{Ke2}, $\Gamma_0$ is amenable if and only if $||h^*h||=(q+1)^2$. This proves $(ii)$ because $||h^*h||=||h||^2$.
\\
($iii$):
The proof of $(iii)$ follows from Theorem 9 of \cite{Leh}.
\end{proof}
Note that in view of Remark~\ref{r:7.3} below, Theorem~\ref{t:1.2} and Theorem~\ref{t:1.3} in the introduction are both special cases of Theorem~\ref{t:LowerboundLeinerSetsUpperBoundAmenability}.

\begin{rem}\label{r:7.3}

  If $Y$ is a symmetric set  (i.e. $Y=Y^{-1}$) it follows immediately from Kesten's Theorem, that $||h||=q+1$ if and only if the subgroup $\langle Y\rangle \subset \Gamma$ generated by $Y$ is amenable. This follows also from $(ii)$ because $\langle Y^2 \rangle \subset \langle Y \rangle $ is a subgroup of $\langle Y \rangle $ of index at most $2$.( Recall that, amenability of a group is always preserved by any subgroup. On the other hand, amenability of a subgroup is preserved by the group if the index is finite.)
\end{rem}

\begin{thm}\label{t:Leinertxirhozetamnconnection}

  Let $Y\subset \Gamma$ be a finite subset in a discrete group $\Gamma$ with $|Y|=q+1$ elements, where $q\ge 1$. Then, with the notation from Section~\ref{s:five}, the following are equivalent:
\begin{itemize}
  \item [($i$)]  $Y$ is a Leinert set.
  \item [($ii$)] $||h_n||_2^2=(q+1)q^{n-1}$ for all $n\in\N$.
  \item [($iii$)] $\xi_n=0$ for all $n\in\N$.
  \item [($iv$)] $\eta_n=0$ for all $n\in\N$.
  \item [($v$)] $\zeta_n=0$ for all $n\in\N$.
  \item [($vi$)]  $m_n=m_n^{(q)}$, where $m_n^{(q)}$ is given by Eq.~(\ref{e:mqn}).
  \item [($vii$)] $\mu_{\tilde h}=\mu^{(q)}$ where
\begin{equation}
  \label{e:muTildeh1} \mu^{(q)}=\frac{q+1}{2\pi}\frac{(4q-t^2)^{1/2}}{(q+1)^2-t^2}\,1_{[-2\sqrt q,2\sqrt q]}(t)\,dt.
\end{equation}
\end{itemize}
\end{thm}
\begin{proof}
  $(i)\iff (iv)$: Recall that
$$\eta_n=\tau(h_{2n})=\tau(\sum_{(s_1,\ldots,s_{2})\in E_{2n}} s_1^{-1}s_2s_3^{-1}\cdots s_{2n-1}^{-1}s_{2n}).$$
Since $\tau(g)=\delta_{g,e}$ for $g\in\Gamma$ and a Leinert set omits the identity the equivalence follows.\\

$(ii)\iff (iii)$ This follows by definition of $\xi_n$.

$(iii)\iff(iv)\iff(v)$: This follows from Proposition~\ref{p:xietazetaRelatedToEachOther}.

$(v)\iff (vi)$: This follows from Proposition~\ref{p:momentsContainingFreePart}.

$(iv)\iff (vii)$: By (\ref{e:rhoeqintQ2ndmuhtilde}) we have for $n\in \N$ that
$$\eta_n=\tau(h_{2n})=\int_{I} Q_{2n}(t)\,d\mu_{\tilde h}(t)$$
 where $I=[-(q+1),q+1]$.
Moreover, since $\mu_{\tilde h}$  is a symmetric measure (cf.~Section~\ref{s:two}), and $Q_m$ is an odd polynomial when $m$ is odd,
\begin{equation*}
  \int_I Q_m(t)d\mu_{\tilde h}(t)=0\qquad m=1,3,5,\ldots.
\end{equation*}

Hence $(iv)$ is equivalent to
\begin{equation}\label{e:36}
  \int_I Q_m(t)d\mu_{\tilde h}(t)=0, \qquad m\in\N.
\end{equation}
Since $\mu_{\tilde h}$ is also a probability measure
\begin{equation}\label{e:37}
  \int_I 1 \,d\mu_{\tilde h}(t)=1.
\end{equation}
Since $\Span\{1,Q,Q_2,\ldots\}$ is the set of all polynomials in $\C[X]$ it follows from Weierstrass' approximation theorem and Riesz Representation theorem that $m_{\tilde h}$ is uniquely determined by (\ref{e:36}) and (\ref{e:37}).
In \cite{Saw} pp. 283-284 (see also \cite{Car}) a sequence of polynomials $(p_n)_{n=0}^\infty$ is defined by
\begin{eqnarray*}
&&p_0(t)=1\\
&&p_1(t)=t\\
&&p_2(t)=t^2-(a+1)\\
&&p_{n+1}(t)=tp_n(t)-a p_n(t)
\end{eqnarray*}
for a fixed number $a\in\N$, and it is proven in \cite{Saw} p. 284 that
\begin{equation}\label{e:38}
  \frac 1{2\pi}\int_{-2\sqrt a}^{2\sqrt a} p_k(t)p_{\ell}(t) \frac{(4a-t^2)^{1/2}}{(a+1)^2-t^2}\,dt=\frac{a^{k-1}}{\lambda_k}\delta_{k\ell},
\end{equation}
where $\lambda_k=1$ for $k\in\N$ and $\lambda_0=\frac{a+1}{a}$.
Letting now $a=q$, then for $n\ge 1$, $p_n(t)$ coincide with our polynomials $Q_n(t)$.
Hence using (\ref{e:38}) first with $k\in\N$ and $\ell=0$, and next with $k=\ell=0$ we get
\begin{equation}
  \frac 1{2\pi}\int_{-2\sqrt q}^{2\sqrt q} Q_n(t) \frac{(4q-t^2)^{1/2}}{(q+1)^2-t^2}\,dt=0\qquad n\in\N
\end{equation}
and
\begin{equation}
  \frac 1{2\pi}\int_{-2\sqrt q}^{2\sqrt q}  \frac{(4q-t^2)^{1/2}}{(q+1)^2-t^2}\,dt=\frac{1}{q+1}\qquad n\in\N.
\end{equation}
Hence, the measure $\mu^{(q)}$ defined in $(\ref{e:muTildeh1})$  satisfies (\ref{e:36}) and (\ref{e:37}). Therefore $(iv)$ is equivalent to that $\mu_{\tilde h}=\mu^{(q)}$.
\end{proof}

\begin{cor}

  The odd moments of $\mu^{(q)}$ are zero.
The even moments of $\mu^{(q)}$ are given by $(m_n^{(q)})_{n=1}^\infty$, i.e.
$$\int_I t^{2n}\, d\mu^{(q)}(t)\,dt=m_n^{(q)}, \qquad n\in\N,$$
where
$$m_n^{(q)}=\binom {2n}n q^n -(q-1)\sum_{k=0}^{n-1} \binom {2n}k q^k.$$
\end{cor}
\begin{proof}
Let $Y$ be a Leinert set with $|Y|=q+1$ elements.
By Theorem~\ref{t:Leinertxirhozetamnconnection},  $\mu_{\tilde h}=\mu^{(q)}$ and $m_n=m_n^{(q)}$. Hence
$$\int_I t^{2n}\, d\mu^{(q)}(t)\,dt=\int_I t^{2n}\, d\mu_{\tilde h}(t)\,dt=m_{n}=m_n^{(q)}.$$
\end{proof}

\subsection{\underline{Connection to the cogrowth coefficients of Cohen and Grigorchuk}}$ $\\
Let $X$ be a finite set of generators of a group $\Gamma$ such that $X\cap X^{-1}=\emptyset$ and $|X|\ge 2$. Let $Y:=X\cup X^{-1}$ and $q=|Y|-1=2|X|-1$.
In \cite{C} and \cite{Gr}, Cohen and Grigorchuk independently  introduced the notion of cogrowth coefficients  $(\gamma_n)_{n=1}^\infty$ for $(\Gamma,X)$, by putting $\gamma_n$ equal to the number of elements in the set
$$\{(s_1,\ldots,s_n)\in Y^{n}\mid s_{i+1}\ne s_i^{-1}(1\le i\le n-1)\text{ and } s_1s_2\ldots s_n=e\}$$
As Cohen puts it, $\gamma_n$ is the number of reduced words in $Y$ of length $n$, which represent the unit element of $\Gamma$.
Since $\tau(g)=\delta_{g,e}$, $g\in\Gamma$ we have
\begin{equation}\label{e:gamman}
  \gamma_n=\sum_{(s_1,\ldots,s_n)\in E_n} \tau(s_1s_2\ldots s_n).
\end{equation}
Note that since $Y=Y^{-1}$, we have $\gamma_{2n}=\eta_n$ according to the definition of the reduced numbers $\eta_n$ in ($\ref{e:xirhozetaN}$).
Cohen proved in pp. 302-303 in \cite{C} that if $|X|\ge 2$ and $X$ does not generate $\Gamma$ freely, then
\begin{equation}\label{e:2.2}
  \gamma=\lim_{n\to\infty} \gamma_{2n}^{\frac 1{2n}}\,\,\,(=\lim_{n\to\infty} \eta_n^{\frac 1 {2n}})
\end{equation}
exists and $\gamma\in (\sqrt q,q]$. Moreover if we let  $h=\sum_{s\in Y} s$,
(assuming still that $X$ does not generate $\Gamma$ freely) by Theorem 3 in \cite{C}:
\begin{equation}\label{e:2.3}
  \gamma+\frac q \gamma = ||h||.
\end{equation}
Since $\gamma >\sqrt q$, it follows that
\begin{equation}
  \gamma=\frac 12 (||h||+\sqrt{||h||^2-4q}).
\end{equation}
We next prove the following extension of the above:

\begin{thm}\label{t:2.6}

  Let $Y\subset\Gamma$ be a finite set with $|Y|=q+1$ elements. Let $h=\sum_{s\in Y s}$ as in Theorem~\ref{t:LowerboundLeinerSetsUpperBoundAmenability},
and let $(\xi_n)$, $(\eta_n)$ $(\zeta_n)$ be as in ($\ref{e:xirhozetaN}$) and let
$$\gamma:=\frac 12(||h||+(||h||^2-4q)^{1/2}).$$
Then
\begin{equation}\label{e:2.5}
  \sqrt q\le \gamma\le q \qquad\text{and}\qquad \gamma+\frac q{\gamma}=||h||.
\end{equation}
Moreover, if $Y$ is not a Leinert set, and $q\ge 2$ then $\gamma>\sqrt q$  and
\begin{equation}\label{e:2.6}
  \gamma=\lim_{n\to\infty} ||h_n||_2^{\frac 1 {n}}=\lim_{n\to\infty} \xi_n^{\frac 1 {2n}}=\lim_{n\to\infty} \eta_n^{\frac 1 {2n}}=\lim_{n\to\infty} \zeta_n^{\frac 1 {2n}}.
\end{equation}
\end{thm}
\begin{proof}
  By Theorem~\ref{t:LowerboundLeinerSetsUpperBoundAmenability} $$2\sqrt q\le ||h||\le q+1.$$
Hence $||h||^2-4q\ge 0$. Let
\begin{eqnarray}
&&  \gamma=\frac 12(||h||+(||h||^2-4q)^{1/2})\\
&&  \gamma'=\frac 12(||h||-(||h||^2-4q)^{1/2}).
\end{eqnarray}
Then
$$\gamma+\gamma'=||h||\qquad \gamma\gamma'=q\qquad \gamma\ge \gamma'\ge 0$$
Hence
\begin{equation}\label{e:gammagesqrtq}
  \gamma \ge \sqrt q \qquad\text{ and } \qquad \gamma+\frac q{\gamma}=\gamma+\gamma'=||h||.
\end{equation}
The function
$$f(t)=t+\frac q t$$
is strictly increasing on the interval $[\sqrt q, \infty)$, and $f(\gamma)=||h||\le q+1=f(q)$.
Hence $\gamma\le q$, which proves (\ref{e:2.5}).
Assume next that $q\ge 2$ and $Y$ is not a Leinert set.
Then by Theorem~\ref{t:LowerboundLeinerSetsUpperBoundAmenability}, $||h||^2-4q>0$. Hence $\gamma>\gamma'$, which shows that $\gamma> \sqrt q$.
We next prove $(\ref{e:2.6})$. Since $||h_n||_2^2\ge\xi_n\ge \eta_n\ge \zeta_n\ge 0$ (Proposition~\ref{p:xietazetabounded}), it is sufficient to show that
\begin{eqnarray}
&&  \label{e:2.7}\limsup_{n\to\infty} ||h_n||_2^{\frac 1{n}}\le \gamma\\
&&  \label{e:2.8}\liminf_{n\to\infty} \zeta_n^{\frac 1{2n}}\ge \gamma.
\end{eqnarray}
Note that (\ref{e:2.7}) follows immediately from the following lemma:

\begin{lem}\label{l:2.7}

For all $n\ge \N$,
$$\zeta_n\le 3\gamma^{2n},\qquad \eta_n\le 3n\gamma^{2n},\qquad \xi_n\le 3n^2\gamma^{2n},\qquad ||h_n||_2^2\le 5 n^2 \gamma^{2n}.$$
\end{lem}
\begin{proof}
  By Proposition~\ref{p:etazetaInTermsOfChebyshevPols} we have
\begin{equation}\label{e:2.9}
\zeta_n=(q-1)+2 q^{n}\int_{\mathsmaller{-(q+1)}}^{\mathsmaller{q+1}} T_{2n}\left(\frac t{2\sqrt q}\right)\,d\mu_{\tilde h}(t)
\end{equation}
where $\mu_{\tilde h}$ is the spectral distribution of $\tilde h=\left(
                                                                   \begin{array}{cc}
                                                                     0 & h^* \\
                                                                     h & 0 \\
                                                                   \end{array}
                                                                 \right)$.
In particular,
\begin{equation}\label{e:2.10}
  \Supp(\mu_{\tilde h})\subset [-||h||,||h||].
\end{equation}
By  formula (2) in page 184 of \cite{HTF2}
\begin{eqnarray*}
&&  T_k(\cos\theta)=\cos(k \theta),\qquad \theta \in [0,\pi]\\
&& T_k(\cosh u)=\cosh(k u),\qquad  u\ge 0.
\end{eqnarray*}
Hence, $|T_k(t)|\le 1$ for $t\in [-1,1]$, $T_k(1)=1$, and $T_k$ is strictly increasing on $[1,\infty)$.
Since also $T_k(-x)=(-1)^k T_k(x)$, we have
$$\max_{|t|\le t_0} |T_k(t)|=T_k(t_0),\qquad \forall t_0\ge 1.$$
Hence by (\ref{e:2.9}) and (\ref{e:2.10})
\begin{equation}
  \zeta_n\le q-1 +2 q^n T_{2n}(\frac{||h||}{2\sqrt q}).
\end{equation}
By (\ref{e:gammagesqrtq}) $t:=\log\frac \gamma{\sqrt q}\ge 0$ is non negative. Then
$$\cosh t=\frac 12 (\frac {\gamma}{\sqrt q}+ \frac{\sqrt q}{\gamma})=\frac{||h||}{2\sqrt q}$$
which implies that
$$T_{2n}(\frac{||h||}{2\sqrt q})=\cosh (2nt)=\frac 12 ((\frac{\gamma}{\sqrt q})^{2n}+(\frac{\sqrt q}{\gamma})^{2n})\le (\frac {\gamma}{\sqrt q})^{2n}.$$
This proves that
$$\zeta_n\le q-1 +2\gamma^{2n}\le 3\gamma^{2n},\qquad n\in\N$$
where the last inequality follows from the inequality $\gamma\ge \sqrt q$.
From Proposition~\ref{p:xietazetaRelatedToEachOther}($iii$) we get, using again $\gamma\ge \sqrt q$ that
\begin{eqnarray*}
 \eta_n&\le& \zeta_n+q\zeta_{n-1}+q^{2}\zeta_{n-2}+\ldots+q^{n-1}\zeta_1 \\
&\le& 3\gamma^{2n}(1+\frac q{\gamma^2}+\ldots+(\frac q{\gamma^2})^{n-1})\\
&\le& 3n \gamma^{2n}.
\end{eqnarray*}
In particular, $\eta_k\le 3n \gamma^{2k}$, $1\le k\le n$.
Applying Proposition~\ref{p:xietazetaRelatedToEachOther}($i$) to this inequality we get in the same way that $\xi_n\le 3n^2\gamma^{2n}$.
Hence
$$||h_n||_2^2=\xi_n+(q+1)q^{n-1}\le \xi_n+2q^n\le 3n^2\gamma^{2n}+2\gamma^{2n}\le 5n^2\gamma^{2n}.$$
\end{proof}

\proof[Proof of Theorem~$\ref{t:2.6}$ (continued)]
We prove now (\ref{e:2.8}). Recall that $\gamma > \sqrt q$, and let $\gamma_1\in(\sqrt q, \gamma)$ be arbitrary.
Let
$$\alpha:=\gamma_1+\frac q{\gamma_1}.$$
Since $f(\gamma):=\gamma+\frac q{\gamma}$ is strictly increasing on $[\sqrt q, \infty)$, we have
\begin{equation}
  2\sqrt q<\alpha < ||h||.
\end{equation}
Since, by (\ref{e:pmnormTinSupport}), $\pm ||h||\in \Supp(\mu_{\tilde h})$
we also have
$\mu_{\tilde h}([\alpha,||h||])>0.$
We have previously seen that $|T_{2n}(t)|\le 1$ for $t\in[-1,1]$ and that $T_{2n}$ is positive and increasing on $[1,\infty)$. Hence
\begin{eqnarray*}
\zeta_n&=&(q-1)+2 q^{n}\int_{\mathsmaller{-(q+1)}}^{\mathsmaller{q+1}} T_{2n}\left(\frac t{2\sqrt q}\right)\,d\mu_{\tilde h}(t)\\
&\ge& (q-1)+2 q^{n}(\int_{\mathsmaller{-2\sqrt q}}^{\mathsmaller{2 \sqrt q}} (-1)\,d\mu_{\tilde h}(t)+\int_{\alpha}^{||h||} T_{2n}\left(\frac t{2\sqrt q}\right)\,d\mu_{\tilde h}(t))\\
&\ge& q-1-2 q^n+ 2q^n T_{2n}(\frac {\alpha}{2\sqrt q}))\mu_{\tilde h}([\alpha,||h||]).
\end{eqnarray*}
Let now $u:=\log \frac {\gamma_1}{\sqrt q}>0$. Then
$$\cosh u=\frac 12 (\frac {\gamma_1}{\sqrt q}+\frac {\sqrt q}{\gamma_1})=\frac{\alpha}{2\sqrt q}.$$
Hence
$$T_{2n}(\frac {\alpha}{2\sqrt q})=\cosh(2nu)=\frac 12 ((\frac{\gamma_1}{\sqrt q})^{2n} +(\frac{\sqrt q}{\gamma_1})^{2n} )\ge \frac 12 \frac {\gamma_1^{2n}}{q^n}$$
which shows that
$$\zeta_n\ge q-1-2q^n+\gamma_1^{2n}\mu_{\tilde h}([\alpha,||h||]).$$
Since $\gamma_1>\sqrt q$ it follows that
$$\liminf_{n\to\infty} \zeta_n^{1/2n}\ge \gamma_1$$
and since $\gamma_1\in(\sqrt q,\gamma)$ was arbitrary we have
$$\liminf_{n\to\infty} \zeta_n^{1/2n}\ge \gamma$$
proving (\ref{e:2.8}).
\end{proof}

\begin{rem}\label{r:6.8}

  If $Y$ is a Leinert set, then $\gamma=\sqrt q$ by Theorem~\ref{t:LowerboundLeinerSetsUpperBoundAmenability}($iii$).
Hence by Theorem~\ref{t:Leinertxirhozetamnconnection}
$$\lim_{n\to\infty} ||h_n||_2^{\frac1{n}}=\sqrt q=\gamma$$
while
$$\lim_{n\to\infty} \xi_n^{\frac 1{2n}}=\lim_{n\to\infty} \eta_n^{\frac 1{2n}}=\lim_{n\to\infty} \zeta_n^{\frac 1 {2n}}=0.$$
\end{rem}

\begin{cor}\label{c:7.9}

  Let $\Gamma$ be a discrete group and let $Y\subset \Gamma$ be a finite set with $|Y|=q+1$ elements $(q\ge 2)$,
such that $Y^{-1}Y$ generates $\Gamma$. Then the following are equivalent:
\begin{itemize}
  \item [$($i$)$] $\Gamma$ is amenable
  \item [$($ii$)$] $\lim\limits_{n\to\infty} ||h_n||_2^{1/n}=q$
  \item [$($iii$)$] $\lim\limits_{n\to\infty} \xi_n^{\frac 1{2n}}=q$
  \item [$($iv$)$] $\lim\limits_{n\to\infty} \eta_n^{\frac 1{2n}}=q$
  \item [$($v$)$] $\lim\limits_{n\to\infty} \zeta_n^{\frac 1{2n}}=q$.
\end{itemize}
 Note as well that $\Gamma$ is amenable if and only if  $ \lim\limits_{n\to\infty} m_n^{\frac{1}{2n}}=q+1$.
\end{cor}
\begin{proof}
  If $\Gamma$ is amenable, then by Theorem~\ref{t:LowerboundLeinerSetsUpperBoundAmenability}($ii$),
$||h||=q+1$ and thus $\gamma=\frac12(||h||+\sqrt{||h||^2-4q})=q$.
Since $q\ge 2$, $||h||=q+1>2\sqrt q$. Hence by Theorem~\ref{t:LowerboundLeinerSetsUpperBoundAmenability}($iii$) $Y$ is not a Leinert set,
and $(ii),(iii),(iv),(v)$ in the corollary follows now from formula~(\ref{e:2.6}) in Theorem~\ref{t:2.6}.
Conversely, if one of the statements $(ii)$, $(iii)$, $(iv)$, or $(v)$ in the corollary holds, then by Remark~\ref{r:6.8}, $Y$ is not a Leinert set, so using again formula~(\ref{e:2.6}) in Theorem~\ref{t:2.6} we have $\gamma=q$, and hence by (\ref{e:2.5}), $||h||=\gamma+\frac {q}{\gamma}=q+1$, which by Theorem~\ref{t:LowerboundLeinerSetsUpperBoundAmenability}($ii$) implies that $\Gamma$ is amenable.
By Proposition~\ref{p:4.1} $$\lim \limits_{n\to\infty} m_n^{\frac{1}{2n}}= ||h||,$$ which proves the last statement in the corollary.

\end{proof}

\appendix

\section{\textbf{On Connectedness of spectra of elements in $C_r^*(F)$}}\label{a:3}

In \cite{Far} Farley proved that the Thompson group $F$ admits a proper affine isometric action on a Hilbert space or equivalently, $F$ has the Haagerup property (cf.~\cite{CCJJV}). Hence by the result of Higson and Kasparov \cite{HK}, $F$ satisfies the Baum-Connes conjecture with coefficients, which in turn implies that $F$ satisfies the Kadison-Kaplansky conjecture, i.e. (since $F$ is torsion free) the reduced $C^*$-algebra  $C_r^*(F)$  has no projections other than $0$ or $1$ (see e.g. \cite{Pusch}). Here projections mean self-adjoint idempotents ($p^2=p=p^*$). This implies that the spectrum of every self-adjoint operator $a\in C_r^*(F)$ is a connected subset of $\R$. The following two propositions are simple applications of this:

\begin{pro}[Case 1]\label{p:A1case1}

Let $h=I+A+B$ and $\tilde h=\left(
                              \begin{array}{cc}
                                0 & h^* \\
                                h & 0 \\
                              \end{array}
                            \right)$.
Then there exists a $\delta\in[0,||h||)$ such that
$$\sigma(\tilde h)=\supp(\mu_{\tilde h})=[-||h||,-\delta]\cup [\delta, ||h||].$$
\end{pro}
\begin{proof}
  Since $h^*h$ is a self-adjoint element of $C_r^*(F)$, its spectrum is a connected subset of $\R$. Hence $\sigma (h^*h)=[\alpha,\beta]$, where
$\alpha=\min(\sigma(h^*h))$ and $\beta=\max(\sigma(h^*h))$. Moreover,
since $h^*h\ge 0$,
$$0\le \alpha\le \beta\qquad\text{and}\qquad \beta=||h^*h||=||h||^2.$$
Since $\mu_{h^*h}$ is the image measure of the symmetric measure $\mu_{\tilde h}$ by the map $t\to t^2$ (see the end of Section~\ref{s:two}), we have
$$\sigma(\tilde h)=\Supp(\mu_{\tilde h})=[-||h||,-\sqrt{\alpha}]\cup [\sqrt{\alpha}, ||h||].$$
Put $\delta=\sqrt{\alpha}$. Then $\delta\in[0,||h||]$.
To see that $\delta< ||h||$, note first that
$$\tau(h^*h)=3$$
because $\{I,A,B\}$ is an orthonormal set with respect to the inner product $\langle a,b\rangle=\tau(b^*a)$. Hence
$$3=\tau(h^*h)=\int_{\alpha}^{\beta} t \, d\mu_{h^*h}(t) \ge \int_{\alpha}^{\beta}\alpha \, d\mu_{h^*h}(t)=\alpha.$$
Hence $\delta=\sqrt \alpha\le \sqrt 3<2\sqrt 2\le ||h||$ by Corollary~\ref{c:1.4}.
\end{proof}

\begin{pro}[Case 2]\label{p:Acase2}

  Let $h=A+A^{-1}+B+B^{-1}$ and $\tilde h=\left(
                                            \begin{array}{cc}
                                              0 & h^* \\
                                              h & 0 \\
                                            \end{array}
                                          \right)$.
\end{pro}
Then
$$\sigma(h)=\Supp(\mu_h)=\Supp(\mu_{\tilde h})=[-||h||,||h||].$$
\begin{proof}
  We know that $\mu_h=\mu_{\tilde h}$ (see end of Section~\ref{s:fourB}).
Hence
$$\sigma(h)=\Supp(\mu_h)=\Supp(\mu_{\tilde h})$$
and from Section~\ref{s:two}, we have
$$\Supp(\mu_{\tilde h})\subset [-||h||,||h||]$$
and
$$\pm ||h||\in \Supp(\mu_{\tilde h}),$$
but since $h=h^*\in C_r^*(F)$, $\sigma(h)$ is a connected subset of $\R$. Hence
$$\Supp(\mu_{\tilde h})=\sigma(h)=[-||h||,||h||],$$
proving the proposition.
\end{proof}

\begin{rem}\label{r:A3}

  By some extra work, it can be proved also for Case 1 that
$$\Supp(\mu_{\tilde h})=[-||h||,||h||].$$
However, in the rest of this section, we shall only need the following immediate corollary of Proposition~\ref{p:A1case1} and Proposition~\ref{p:Acase2}:

\end{rem}

\begin{cor}\label{c:A4}

  In both Case 1 and Case 2, $\sigma(\tilde h)=\Supp(\mu_{\tilde h})$ is an infinite subset of $\R$ without isolated points.
\end{cor}

Let $H$ be a Hilbert space, and let $K(H)\subset B(H)$ be the set of compact operators on $H$ and let $\rho$ denote the quotient map $\rho:B(H)\to B(H)/K(H)$.
Then the essential spectrum $\sigma_{\mathrm{ess}}(T)$ of an element $T\in B(H)$ is the spectrum $\sigma(\rho(T))$ of $\rho(T)$, see e.g. p. 30 in \cite{Mur}. Note that $\sigma_{\mathrm{ess}}(T)\subset \sigma(T)$ and for every $K\in K(H)$, $\sigma_{\mathrm{ess}}(T+K)=\sigma_{\mathrm{ess}}(T)$. For a self-adjoint operator $S=S^*\in B(H)$, it is easy to see that $\lambda\in \sigma_{\mathrm{ess}}(S)$ if and only if  for every $\varepsilon>0$, the spectral projection $1_{(\lambda-\varepsilon,\lambda+\varepsilon)}(S)$ of $S$ is infinite dimensional. The following result is well known (cf.~Theorem VII.10 in \cite{RS}).

\begin{pro}\label{p:A5}

  Let $S=S^*\in B(H)$. Then $\lambda\in \sigma(S)\backslash\sigma_{\mathrm{ess}}(S)$ if and only if the following two conditions hold
\begin{itemize}
  \item [$($i$)$] $\lambda$ is an isolated point in $\sigma(S)$.
  \item [$($ii$)$] $\lambda$ is an eigenvalue of $S$ and $\mathrm{dim}(\ker(S-\lambda I))<\infty$.
\end{itemize}
\end{pro}

\begin{cor}\label{c:A6}

  If $S=S^*\in B(H)$ and $\sigma(S)$  has no isolated points then $\sigma_{\mathrm{ess}}(S)=\sigma(S)$.
\end{cor}

\begin{pro}\label{p:A7}

  Let $M\in B(\ell^2(\N_0))$ be an operator of the form
 \begin{equation}
    M=\left(
     \begin{array}{ccccc}
       0 & \alpha_1 & 0 & 0 &\cdots \\
       \alpha_1 & 0 & \alpha_2 & 0 &\cdots \\
       0 &  \alpha_2 &  0 & \alpha_3 &\cdots\\
       0 & 0 &  \alpha_3 &  0 &\cdots\\
       \vdots & \vdots & \vdots &\vdots & \ddots \\
     \end{array}
   \right)
  \end{equation}
where $\alpha_n\ge 0$ for all $n\in \N$ and assume that $\sigma(M)$ has no isolated points. Then
\begin{equation}\label{e:limsupimprovement}
  \liminf_{n\to\infty}(\alpha_{n-1}+\alpha_n)\le ||M||\le \limsup_{n\to\infty}(\alpha_{n-1}+\alpha_n).
\end{equation}
\end{pro}
\begin{proof}
  The inequality on the left hand side follows from the proof of Proposition~\ref{p:4.4} (without any assumption on $\sigma(M)$).
Put
\begin{equation}
    N_n=\left(
     \begin{array}{ccccc}
       0 & \alpha_{n+1} & 0 & 0 &\cdots \\
       \alpha_{n+1} & 0 & \alpha_{n+2} & 0 &\cdots \\
       0 &  \alpha_{n+2} &  0 & \alpha_{n+3} &\cdots\\
       0 & 0 &  \alpha_{n+3} &  0 &\cdots\\
       \vdots & \vdots & \vdots &\vdots & \ddots \\
     \end{array}
   \right),\qquad n\in\N.
\end{equation}
Then
$||N_n||=||Q_nMQ_n||\le ||M||$, where $Q_n$  denotes the projection on $\Span\{\delta_k\mid k\ge n\}$, and $(\delta_k)_{k=0}^\infty$  is the standard basis for $\ell^2(\N_0)$.
But since $Q_n M Q_n$ is a compact perturbation of $M$, $\sigma_{\mathrm{ess}}(Q_nMQ_n)=\sigma_{\mathrm{ess}}(M)$ so by Corollary~\ref{c:A6}
$$\sigma(Q_n M Q_n)\supset \sigma_{\mathrm{ess}}(Q_n MQ_n)=\sigma_{\mathrm{ess}}(M)=\sigma(M),$$
and since the norm of a self-adjoint operator is equal to its spectral radius, also
$$||Q_n M Q_n||\ge ||M||,\qquad n\in \N,$$
so altogether,
$$||N_n||=||Q_nMQ_n||=||M||,\qquad n\in\N.$$
By the proof of Proposition~\ref{p:4.4},
$$||M||=||N_n||\le \sup\{\alpha_{k-1}+\alpha_k\mid k\ge n+2\}.$$
Hence
$$||M||\le \limsup_{n\to\infty}(\alpha_{n-1}+\alpha_n).$$
\end{proof}

\begin{cor}\label{c:A8}

  Let $h=I+A+B$ (case 1) or $h=A+A^{-1}+B+B^{-1}$ (case 2) and let
 \begin{equation}
    M=\left(
     \begin{array}{ccccc}
       0 & \alpha_1 & 0 & 0 &\cdots \\
       \alpha_1 & 0 & \alpha_2 & 0 &\cdots \\
       0 &  \alpha_2 &  0 & \alpha_3 &\cdots\\
       0 & 0 &  \alpha_3 &  0 &\cdots\\
       \vdots & \vdots & \vdots &\vdots & \ddots \\
     \end{array}
   \right)
  \end{equation}
be the operator in $B(\ell^2(\N_0))$ built from $\mu_{\tilde h}$ as in Section~\ref{s:four}, then
\begin{equation}\label{e:90}
  \liminf_{n\to\infty}(\alpha_{n-1}+\alpha_n)\le ||M||\le \limsup_{n\to\infty}(\alpha_{n-1}+\alpha_n).
\end{equation}
\end{cor}
\begin{proof}
By the proof of Proposition~\ref{p:4.3}, $M$ is unitarily equivalent to the multiplication operator $m_t$ on $L^2(\mu_{\tilde h})$ given by $f(t)\mapsto tf(t)$, $f\in L^2(\mu_{\tilde h})$.
Hence $\sigma(M)=\sigma(m_t)=\Supp(\mu_{\tilde h})=\sigma(\tilde h)$. Formula~(\ref{e:90}) now follows from Proposition~\ref{p:A7} and Corollary~\ref{c:A4}.
\end{proof}

\section{\textbf{A number theoretical Test}}\label{a:1}

This section contains a proof of the test for computational errors mentioned at the end of Section~\ref{s:three}.
Let $\Gamma$ be a group and $Y\subset \Gamma$ a finite subset of $\Gamma$ with $|Y|=q+1$ elements $(q\ge 2)$.
Moreover, let $m_n$, $||h_||_2^2$, $\xi_n$, $\eta_n$, $\zeta_n$ be the sequences of numbers introduced in Sections~\ref{s:three} and \ref{s:five}.
Then the following holds:

\begin{thm}\label{t:B1}

\begin{itemize}
  \item [$($i$)$] $\xi_1=\eta_1=\zeta_1=0$ and $m_1=||h_1||_2^2=q+1$.
  \item [$($ii$)$]  For $n\ge 2$ the numbers $\xi_n$, $\eta_n$, $\zeta_n$, $||h_n||_2^2$ are all even integers, while the $m_n$ numbers have the same parity as the number $q+1$.
  \item [$($iii$)$]   Let $\mu:\N\to\{-1,0,1\}$ denote the M\"obius function, i.e. $\mu(1):=1$ and for $n\ge2$,
$$\mu(n):=\left\{
                  \begin{array}{ll}
                    (-1)^k & \hbox{if $n=p_1p_2\cdots p_k$ (product of $k$ distinct prime numbers),} \\
                    0 & \hbox{otherwise.}
                  \end{array}
                \right.$$
If $\Gamma$ is torsion free, then the numbers
\begin{equation}
  \zeta'_n:=\sum_{\substack{d|n}} \mu(\frac nd)\zeta_d, \qquad n\in\N
\end{equation}
 are non-negative integers divisible by $2n$. In particular,
if $n=p$ is a prime number, then $\zeta_p=\zeta'_p$ and hence $\zeta_p$ is divisible by $2p$.
\end{itemize}
\end{thm}
\begin{proof}
$(i)$:  By definition
$$\eta_1=|\{(s_1,s_2)\in Y^2\mid s_1\ne s_2, s_1^{-1}s_2=e\}|=0.$$
Hence, by Propositions~\ref{p:xietazetaRelatedToEachOther} and (\ref{e:xirhozetaN}), $\xi_1=\zeta_1=0$ and $m_1=||h_1||_2^2=q+1$.\\
$(ii)$: Recall that for $n\in\N$:
$$\tilde E_{2n}=\{(s_1,\ldots,s_{2n})\in Y^{2n}\mid s_1\ne s_2\ne \ldots\ne s_{2n}\ne s_1\}.$$
and  $\zeta_n=|R_n|$ where
$$R_n:=\{(s_1,\ldots,s_{2n})\in \tilde E_{2n}\mid s_1^{-1}s_2\cdots s_{2n-1}^{-1}s_{2n}=e\}.$$
Define the reversing map $\sigma:\tilde E_{2n}\to\tilde E_{2n}$ by
$$\sigma(s_1,\ldots,s_{2n}):=(s_{2n},\ldots,s_1).$$
Then $\sigma^2=id$ and $\sigma(R_n)=R_n$.
Moreover, for each $s=(s_1,\ldots,s_{2n})\in \tilde E_{2n}$, $\sigma(s)\ne s$, because $s_1\ne s_{2n}$.
Hence, all the orbits in $\tilde E_{2n}$  under the action of the group $\{id, \sigma\}\cong \Z_2$ have size 2.
Since $\sigma(R_n)=R_n$, $\zeta_n=|R_n|$ is an even number.
Thus, by Proposition~\ref{p:xietazetaRelatedToEachOther}, $\eta_n$ and $\xi_n$ are also even. Hence for $n\ge 2$,
$$||h_n||_2^2=\xi_n+(q+1)q^{n-1}$$ is also even.
The last statement in $(ii)$ follows from (\ref{e:mqn}) and (\ref{e:mnfromzetas}).\\
($iii$): Let the reversing map $\sigma:\tilde E_{2n}\to \tilde E_{2n}$  and $R_n$ be as in the proof of ($ii$), and define $\rho:\tilde E_{2n}\to \tilde E_{2n}$ by
$$\rho(s_1,\ldots, s_{2n}):=(s_3,s_4,\ldots,s_{2n}, s_1, s_2).$$
Then clearly $\rho(R_n)=R_n$, $\rho^n=id$, $\sigma\rho\sigma=\rho^{-1}$, and the group $G$ of transformations of $\tilde E_{2n}$ generated by $\sigma$ and $\rho$ is equal to
$$G=H\sqcup \sigma H$$
where $H=\{id,\rho,\rho^2,\ldots,\rho^{n-1}\}\cong \Z_n$.
Moreover, $G\cong \Z_n\rtimes\Z_2$, the dihedral group with $2n$ elements. We prove next the following claim:
For each $s=(s_1,\ldots,s_{2n})\in \tilde E_{2n}$, the stabilizer group
$$G_s=\{g\in G\mid gs=s\}$$
is contained in $H$.\\
To prove the claim, we just have to show that $\sigma H\cap G_s=\emptyset$.
From the proof of $(ii)$, we already know that $\sigma\not\in G_s$. Let now $k\in\{1,\ldots, n-1\}$ and assume that $\sigma\rho^k\in G_s$.
Then
$$(s_{2k},s_{2k-1},\ldots, s_1,s_{2n},s_{2n-1},\ldots,s_{2k+1})=(s_1,\ldots,s_{2n}).$$
By comparing the $k$'th terms of the two tuples, we get $s_{k+1}=s_k$, which is a contradiction. Hence $\sigma H\cap G_s=\emptyset$ proving the claim.

Since $G_s$ is a subgroup of $H\cong \Z_n$, $G_s$ is of the form
$$H_d=\{\rho^k\mid \text{ $k$ is a multiple of $d$}\},$$
where $d\in\{1,\ldots, n\}$ is a divisor of $n$. Note that $|H_d|=\frac nd$.
In particular $H_n=\{id\}$.
Assume now that $\Gamma$ is torsion free and write
$$R_n=\bigsqcup_{d\mid n} R_n^d,$$
where $R_n^d=\{s\in R_n\mid G_s=H_d\}.$
If $s\in R_n^d$, for a divisor $d$ of $n$, for which $d<n$, then
\begin{equation}\label{e:apendixBestar}
  s=(s_1,\ldots,s_{2d},s_1,\ldots,s_{2d},\ldots, s_1,\ldots,s_{2d})
\end{equation}
where $(s_1,\ldots, s_{2d})$ is repeated $\tfrac nd$ times, and $\rho^k s\ne s$ for $k=1,\ldots,d-1$.
Since $s\in R_n$,
$$(s_1^{-1}s_2\cdots s_{2d-1}^{-1}s_{2d})^{n/d}=e,$$
and since $\Gamma$ is torsion free, also
$$s_1^{-1}s_2\cdots s_{2d-1}^{-1}s_{2d}=e.$$
Together with $\rho^ks\ne s$ for $k=1,\ldots,d-1$, this shows that $(s_1,\ldots, s_{2d})\in R_d^d$. Conversely,
if $(s_1,\ldots, s_{2d})\in R_d^d$, then $s$ given by (\ref{e:apendixBestar}) will be in $R_n^d$. Hence $R_n^d$ and $R_d^d$ have the same number of elements proving that
$$\zeta_n=|R_n|= \sum_{d\mid n} |R_d^d|.$$
Hence, by the M\"obius inversion formula, (cf.~Theorem 2.9 in \cite{Apostol}) we have
$$|R_n^n|=\sum_{d\mid n} \mu(\frac n d)\zeta_d.$$
Recall that
$$R_n^n=\{s\in R_n: G_s=H_n=\{id\}\}.$$
Hence the orbit $G.s$ has $2n$ elements for all $s\in R_n^n$, which implies that $|R_n^n|$ is divisible by $2n$ proving ($iii$).
\end{proof}

\subsection*{\underline{Acknowledgments}}
We are very grateful to James Avery at the Center for Experimental Mathematics at University of Copenhagen for running some of our computer programs on a supercomputer.

%

\Addresses

\begin{thebibliography}{0}


\bibitem{AO}
C.A. Akemann and P.A. Ostrand.
\emph{Computing norms in group $C^*$-algebras.}
Amer. J. Math. 98, 1015-1047
(1976).

\bibitem{Apostol}
T. M. Apsotol.
\emph{Introduction to Analytic Number Theory}.
Springer-Verlag
(1976).


\bibitem{GubaArzhantsevaLustingPreaux}
G.N. Arzhantseva, V.S. Guba, M. Lustig and J.-P. Pr\'eaux.
\emph{Testing Cayley graph densities}.
Annales Math. Blaise Pascal 15, 233-286
(2008).

\bibitem{BelkBrownForestDiagrams}
  J. Belk, K. Brown.
  \emph{Forest diagrams for elements of Thompson's group $F$.}
   Internat. J. Algebra Comput. 15 no. 5-6, 815-850,
  (2005).


\bibitem{BrinS}
M. G. Brin, C. C. Squier,
\emph{Groups of piecewise linear homeomorphisms of the real line.}
Invent. Math. 79, 485-498.
(1985).

\bibitem{BrownOzawa}
N. P. Brown and N. Ozawa.
\emph{$C^*$-algebras and Finite-Dimensional Approximations}.
Graduate Studies in Mathematics, Vol. 88, Amer. Math. Soc.
(2008).


\bibitem{BurilloClearyWiest}
J. Burillo, S. Cleary, and B. Wiest.
 \emph{Computational exploration in Thompson's group $F$.}
Geometric group theory, 21-35, Trends Math. Birkh\"auser, Basel
(2007).


\bibitem{CFP}
J. W. Cannon, W. J. Floyd, and W. R. Parry.
\emph{Introductory notes on Richard Thompson's groups.}
Enseign. Math. (2), 42(3-4):215-256
(1996).


\bibitem{Car}
P. Carti\'er.
\emph{Harmonic analysis on trees.}
In ``Harmonic analysis on homogeneous spaces."
Proc. Sym. Pure Math. Vol. 26. AMS
(1973).


\bibitem{CCJJV}
P.-A. Cherix, M. Cowling, P. Jolissaint, P. Julg and A. Valette.
\emph{Groups with the Haagerup Property, Gromov's  a-T-amenability.}
 Progress in Mathematics 197, Birkha\"user
(2001).


\bibitem{C}
J.M. Cohen.
\emph{Cogrowth and Amenability of Discrete groups.}
J. Funct. Analysis 48, 301-309
(1982).


\bibitem{ElderFusyRechnitzerCountingGeodesics}
  M. Elder, E. Fusy, A. Rechnitzer.
  \emph{Counting elements and geodesics in Thompson's group $F$.}
  J. Algebra 324, no. 1, 102-121.
  (2010).


\bibitem{ElderRechnitzerWongCogrowth}
   M Elder, A. Rechnitzer,  T. Wong.
   \emph{On the cogrowth of Thompson's group $F$.}
    Groups Complex. Cryptol. 4. no. 2, 301-320
   (2012).


\bibitem{ElderRechnitzerJanse}
M. Elder, A. Rechnitzer and E.J. Janse van Reusburg.
\emph{ Random sampling of trivial words in finitely presented groups.}
arXiv:math/1312.5722v2
(2013).


\bibitem{Far}
D. S. Farley,
\emph{Proper isometric action of Thompson's groups on Hilbert spaces.}
 IMRN, Int, Math. Research Notices 45, 2409-2414
(2003).


\bibitem{Gr}
R.I. Grigorchuk.
\emph{Symmetrical random walks on discrete groups.}
In ``Multicomponent Random Systems", 285-325 (ed. R.L. Dobrushin and Ya. G. Sinai),
New York-Basel
(1980).


\bibitem{GubaCayLeyGraphF}
  V. S. Guba.
  \emph{On the Properties of the Cayley Graph of Richard Thompson's Group $F$.}
   Internat. J. Algebra Comput. 14, no. 5-6, 677-702.
  (2004).


\bibitem{HaagerupKnydby2013}
U. Haagerup, S. Knudby.
\emph{A L\'evy-Khintchin formula for free groups.}
Proc. Amer. Math. Soc (2014). DOI: 10.1090/S0002-9939-2014-12466-X
(2013).


\bibitem{HO}
U. Haagerup, K. K. Olesen.
\emph{The thompson group $T$ and $V$ are not inner amenable.}
In preparation.
(2014)


\bibitem{HTF2}
\emph{Higher Transcendental Functions, Vol II.}
(ed. A. Erd\' elyi), Mc. Graw-Hill
(1953).


\bibitem{HK}
N. Higson and G. Kasparov,
\emph{E-theory and KK-theory for groups which act properly and isometrically on Hilbert space.}
Invent. Math. 144, 23-74
(2001).


\bibitem{KadRin}
R. V. Kadison, J. Ringrose,
\emph{Fundamentals of the theory of operator algebras. Vol. 2}.
American Mathematical Society.
(1986).


\bibitem{Ke2}
H. Kesten.
\emph{Full Banach mean values on countable groups.}
Math. Scand. 7, 146-156
(1959).


\bibitem{Kesten}
H. Kesten.
\emph{Symmetric random walks on groups.}
 Trans. Amer. Math. Soc. 92, 336-354
(1959).


\bibitem{Leh}
F. Lehner.
\emph{A characterization of the Leinert Property.}
Proc. Amer. Math. Soc. 125, 3423-3431
(1997).


\bibitem{Lei}
M. Leinert.
\emph{Faltungsoperatoren auf gewissen diskreten Gruppen.}
Studia Math. 52, 149-158
(1974).


\bibitem{Monod}
N. Monod.
\emph{Groups of piecewise projective homeomorphisms.}
Proc. Natl. Acad. Sci. USA 110, no. 12, 4524-4527
(2013)


\bibitem{Mur}
G. J. Murphy.
\emph{$C^*$-algebras and operator theory.}
Academic Press,
(1990).


\bibitem{GKP}
Gert K. Pedersen.
\emph{Analysis Now.}
Graduate Texts in Mathematics 118.
Springer
(1988).


\bibitem{Pusch}
M Puschnigg.
\emph{The Kadison-Kaplansky conjecture for word hyperbolic groups.}
Invent. Math. 149, 153-194
(2002).


\bibitem{RS}
M. Reed and B. Simon.
\emph{Methods of Modern Mathematical Physics I: Functional Analysis.}
 Academic Press
(1980).


\bibitem{Saw}
S. Sawyer.
\emph{Isotropic random walks on trees.}
Zeithschirft Wahrscheirnl\'echeitsth. 42, 279-292
(1978).


\bibitem{S}
G. Szego.
\emph{Orthogonal Polynomials.}
 3rd edition, Colloq. Publ. Vol. 23.  Amer. Math. Soc.
 (1967)





\end{thebibliography}
\end{document}